\documentclass{siamltex}
\usepackage[utf8]{inputenc}
\usepackage[english]{babel}
\usepackage{amsmath}
\usepackage{amsfonts}
\usepackage{amssymb}

\usepackage[vlined,linesnumbered]{algorithm2e}
\usepackage{graphicx}
\usepackage{hyperref}
\usepackage{mathabx}
\usepackage{paralist}
\usepackage{qtree}
\usepackage{subcaption}
\usepackage{url}

\DeclareMathOperator{\argmax}{argmax}
\DeclareMathOperator{\cheb}{\bf cheb}
\DeclareMathOperator{\rand}{\bf rand}
\DeclareMathOperator{\rt}{root}
\DeclareMathOperator{\sons}{sons}
\DeclareMathOperator{\spn}{span}

\begin{document}
\title{Finding entries of maximum absolute value in low-rank tensors}
\author{
  Lars Grasedyck\thanks{Institut f\"ur Geometrie und Praktische Mathematik, 
    RWTH Aachen, Templergraben 55, 52056 Aachen, Germany. 
    Email: {\tt \{lgr,loebbert\}@igpm.rwth-aachen.de}. The authors gratefully acknowledge support by the
    DFG priority programme 1648 (SPPEXA) under grant GR-3179/4-2 and 
    DFG priority programme 1886 under grant GR-3179/5-1.
  }
  \and
  Lukas Juschka\thanks{Email: {\tt lukas.juschka@rwth-aachen.de}}
  \and
  Christian L\"obbert$^*$
}
\date{}

\maketitle
\sloppy
\begin{abstract}
We present an iterative method for the search of extreme entries in low-rank tensors which is based
on a power iteration combined with a binary search.
In this work we use the HT-format for low-rank tensors but other low-rank formats 
can be used verbatim.
We present two different approaches to accelerate the basic power iteration:
an orthogonal projection of Rayleigh-Ritz type, as well as an acceleration of the
power iteration itself which can be achieved due to the diagonal structure of the underlying
eigenvalue problem.
Finally the maximizing index is determined by a binary search based on the proposed iterative
method for the approximation of the maximum norm.
The iterative method for the maximum norm estimation inherits the linear complexity 
of the involved tensor arithmetic in the HT-format w.r.t. the tensor order, 
which is also verified by numerical tests.
\end{abstract}

\section{Introduction}\label{section:introduction}

Scientific computing with functions, vectors and matrices in low rank formats allows
us to tackle even high-dimensional and extreme scale problems, provided the underlying
mathematical objects allow for a low rank approximation. In this article we focus on the 
task of postprocessing low rank matrices and tensors. Our guiding problem is to find the 
entry of maximum absolute value in a low rank object without accessing every entry. Such 
a task is the discrete analogue to finding global optima and thus of general interest. 
While this is straight-forward for rank one, it is surprisingly difficult for ranks beyond one. 
The seminal idea for finding the maximum absolute value efficiently stems from the PhD thesis 
of Mike Espig \cite{Espig08} where the problem is reformulated as a diagonal (extreme scale)
eigenvalue problem. The latter can e.g. be solved by a simple power iteration, and in case that 
there is a unique maximizer, the corresponding eigenvector is rank one. 
Hence, it makes sense to approximate the eigenvector by low rank 
(cf. \cite{Litvinenko19} for a recent approach independently of this article).
However, if several 
entries are of similar size, then the power iteration can be arbitrarily slow, the
iterates cannot be approximated by low rank and a truncation to low rank may destroy
convergence altogether. 

In this article we provide much more advanced approaches for finding the maximum by subspace 
acceleration and a squaring trick. Without truncation we are able 
to prove a uniformly good convergence rate of $1/2$ for computing the $\|\cdot\|_{\infty}$-norm
and based on this we devise a divide and conquer strategy to find the corresponding entry.

The article is structured as follows: In Section~\ref{section:tensorformats} we introduce the HT-format followed by the 
problem for rank one or elementary tensors in Section~\ref{section:rank1}. In Section~\ref{section:power} we summarize the power 
iteration approach and give examples that underline the involved difficulties. 
In Section~\ref{section:rayritz} we 
provide the subspace acceleration and in Section~\ref{section:squaring} the squaring trick. Finally, in Section~\ref{section:finding} we 
explain how one can determine the entry of maximum absolute value based on the so far introduced 
iterative methods. We conclude by giving numerical examples in 
Section~\ref{section:experiments} that show the benefits and limitations of the 
algorithms.

\section{Low-rank formats for tensors}\label{section:tensorformats}
We denote the entry of an order $d$ tensor $\mathbf{a} \in \mathbb R^I, I = \bigtimes_{\mu = 1}^d I_\mu$, by  
\begin{equation*}
\mathbf a[\mathbf i] =\mathbf a[i_1,\ldots ,i_d] \;,\qquad \mathbf i = (i_1,\ldots ,i_d)\in I.
\end{equation*}
For larger tensor order $d \gg 2$ a full (entrywise) representation is not feasible
(cf. the curse of dimensionality \cite{Bellman03,Bellman61}) and additional structure is required.
In the following we use low rank tensors that generalize the matrix rank.
\begin{definition}[Elementary tensors]\label{definition:elementaryTensors}
A tensor $\mathbf a\in\mathbb{R}^I$ is called elementary tensor or rank $1$ tensor if
\begin{equation}\label{equation:elementaryTensor}
\mathbf a = \bigotimes_{\mu = 1}^d u^{(\mu)} = u^{(1)}\otimes\cdots\otimes u^{(d)},
\qquad
u^{(\mu)}\in\mathbb{R}^{I_\mu},\quad \mu = 1,\ldots , d.
\end{equation}
\end{definition}
The CP-rank (Canonical Polyadic) \cite{Hitchcock27} of a tensor $\mathbf a\in\mathbb{R}^I$ is defined
as the minimum number $\mathrm{rank}(a)=r\in\mathbb{N}_0$ such that 
\begin{equation}\label{equation:cpDecomposition}
\mathbf a = \sum_{k=1}^r U^{(1)}[\bullet,k]\otimes \cdots \otimes U^{(d)}[\bullet,k],
\qquad
U^{(\mu)}\in\mathbb{R}^{I_\mu\times r},\quad \mu = 1,\ldots ,d.
\end{equation}
For computational purposes the minimum number $r$ of terms is not necessary, instead any representation
of the form \eqref{equation:cpDecomposition} with reasonably small $r$ is sufficient. Such a representation is
called CP-format with representation rank $r$. The corresponding set is defined by
\begin{equation}\label{equation:setRr}
\mathcal{R}_r := \{\mathbf a\in\mathbb{R}^I\; :\; \rank(\mathbf a)\leq r\}. 
\end{equation}
These sets are numerically difficult and thus a different type of rank is commonly used
which is based on matricizations.
\begin{definition}[$t$-matricization]\label{definition:matricization}
Let $\mathbf a\in\mathbb R^I$. We denote 
\begin{equation}\label{equation:notationSubIndexSets}
I_t :=\bigtimes_{\mu\in t} I_\mu\qquad\mbox{and}\qquad 
I_{[t]} :=\bigtimes_{\mu\notin t} I_\mu,
\qquad t\subseteq \{1,\ldots ,d\}
\end{equation}
and we use the short notation 
\begin{equation}\label{equation:shortNotationIndices}
\mathbf i_t = (i_\mu)_{\mu\in t}, i_t\in I_t\qquad 
\mathbf i_{[t]} = (i_\mu)_{\mu\notin t}, i_{[t]}\in I_{[t]}. 
\end{equation}
Then the $t$-matricization $\mathcal{M}_t(\mathbf a)\in\mathbb{R}^{I_t\times I_{[t]}}$ of 
$\mathbf a$ is defined by
\begin{equation*}
\mathcal{M}_t(\mathbf a)[ \mathbf i_t,\, \mathbf i_{[t]}] :=
\mathbf a[\mathbf i],\qquad
\mathbf i = (i_1,\ldots ,i_d)\in I,
\end{equation*}
i.e. a rearrangement of $\mathbf a$ into a matrix with row index set $I_t$ and column index set
$I_{[t]}$.
\end{definition}
For the CP-format, instead of all tensor entries $\mathbf a[i_1,\ldots ,i_d]$ of some tensor
$\mathbf a\in\mathbb{R}^I$, the matrices $U^{(\mu)}\in\mathbb{R}^{I_\mu\times r}$ from the
right-hand side of \eqref{equation:cpDecomposition} are stored.
Assuming all tensor directions to be of the same size $n = \#I_\mu$, this sums up to a storage
complexity of $dnr$, which is linear in the tensor order $d$.
Despite of this desirable linear complexity in $d$, the CP-format has two drawbacks:
on the one hand, the CP-rank is in general NP-hard to determine, cf. \cite{Hastad90}, 
on the other hand, the set \eqref{equation:setRr} is not closed for tensors of higher order
$d\geq 3$, cf. \cite{deSilva08,Ha12}. 
This fact can render the numerical treatment of tensors in CP-format difficult.

For the Tucker format the corresponding set 
is closed, cf. \cite{Ha12}, but the storage complexity w.r.t. the tensor order $d$ is exponential.

The above mentioned complexities and shortcomings transfer to many arithmetic operations, such as 
dot products or Hadamard products, which are involved in the algorithms presented in this article.
All of the shortcomings can be resolved by using the Hierarchical Tucker format (HT-format) introduced
in the following. 

\subsection{The Hierarchical Tucker format}\label{subsection:htFormat}

An HT-format for a tensor $\mathbf a\in\mathbb{R}^I$, $I = \bigtimes_{\mu = 1}^d I_\mu$,
is based on a hierarchy of all tensor directions $D:= \{1,\ldots ,d\}$, which can be 
expressed by a binary tree, referred to as dimension tree, cf. \cite{Gr10,Ha12}:

\begin{definition}[Dimension tree]
For a tensor space $\mathbb R^I$, $I = \bigtimes_{\mu = 1}^d I_\mu$, by $D = \{1,\ldots ,d\}$
we denote the set of all tensor directions. A corresponding dimension tree $T_D$ is a 
tree with the following properties:
\begin{enumerate}[(a)]
\item all vertices of $T_D$ are non-empty subsets of $D$,
\item the root vertex of $T_D$ is the set $D$: $\rt(T_D) = D$,
\item all vertices $t\in T_D$ with $\# t > 1$ have two disjoint successors $\sons(t) = \{s_1,s_2\}$ 
with
\begin{equation*}
t = s_1 \dot\cup s_2\qquad\mbox{(disjoint union)}.
\end{equation*}
\end{enumerate}
\end{definition}
A dimension tree $T_D$ for a tensor order $d\in\mathbb{N}$, $D = \{1,\ldots ,d\}$, is thus a
binary tree with singletons $\{\mu\}$, $\mu = 1,\ldots ,d$, at the leaf vertices. 
Two examples of dimension trees for tensor order $d=4$ are shown in 
Figure~\ref{figure:dimensionTrees4}.
\begin{figure}
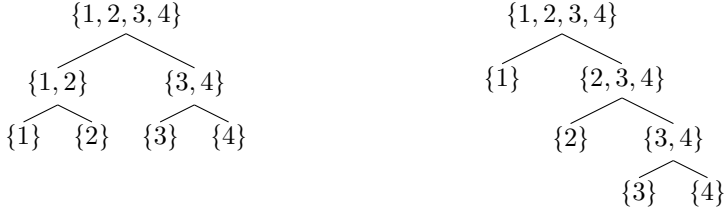

\begin{subfigure}[b]{0.48\textwidth}
\Tree
[.$\{1,2,3,4\}$
    [.$\{1,2\}$
        $\{1\}$
        $\{2\}$
    ].$\{1,2\}$
    [.$\{3,4\}$
        $\{3\}$
        $\{4\}$
    ].$\{3,4\}$
].$\{1,2,3,4\}$
\end{subfigure}
\begin{subfigure}[b]{0.48\textwidth}
\Tree
[.$\{1,2,3,4\}$
    $\{1\}$
    [.$\{2,3,4\}$
        $\{2\}$
        [.$\{3,4\}$
            $\{3\}$
            $\{4\}$
        ].$\{3,4\}$
    ].$\{2,3,4\}$
].$\{1,2,3,4\}$
\end{subfigure}
\caption{Two different dimension trees for tensor order $d=4$.}
\label{figure:dimensionTrees4}
\end{figure}

The HT-format is again defined by matricizations, cf. Definition~\ref{definition:matricization},
where the subsets $t\subseteq \{1,\ldots ,d\}$ are the subsets $t\in T_D$ of the underlying 
dimension tree $T_D$ which determines the structure of the HT-format:
\begin{definition}[Hierarchical Tucker format (HT-format)]\label{definition:htFormat}
Let $\mathbf a\in\mathbb{R}^I$, $I = \bigtimes_{\mu = 1}^d I_\mu$, be a tensor and
$T_D$, $D = \{1,\ldots ,d\}$, a dimension tree. Furthermore, for all $t\in T_D$ let
\begin{equation}\label{equation:tMatricizationDecomposition}
\mathcal{M}_t(\mathbf a) = U^{(t)}(V^{(t)})^\top,\qquad
U^{(t)}\in\mathbb{R}^{I_t\times r_t},\qquad
V^{(t)}\in\mathbb{R}^{I_{[t]}\times r_t},
\end{equation}
$I_t$, $I_{[t]}$ as in \eqref{equation:notationSubIndexSets},
be a decomposition of the $t$-matricization of $\mathbf a$, where the $U^{(t)}$, $t\in T_D$,
are referred to as frames: the columns of $U^{(t)}$ span the range of $\mathcal{M}_t(\mathbf a)$,
especially for $t =D$ we assume \eqref{equation:tMatricizationDecomposition} to be chosen such
that
\begin{equation}\label{equation:rootMatricization}
\mathbf a = \mathcal{M}_D(\mathbf a) = U^{(D)}
\end{equation}
holds, i.e. $V^{(D)} = 1\in\mathbb{R}^{1\times 1}$.
A Hierarchical Tucker format (HT-format) for $\mathbf a$ is then defined by
\begin{enumerate}[(a)]
\item the above matrices $U^{(\{\mu\})}\in\mathbb{R}^{I_\mu \times r_{\{\mu\}}}$ for all leaf vertices $\{\mu\}$, $\mu = 1,\ldots ,d$, of $T_D$ and \label{enumerate:htFormatLeaves}
\item transfer tensors $\mathbf b^{(t)}\in\mathbb{R}^{r_t\times r_{s_1}\times r_{s_2}}$ at all
\label{enumerate:htFormatNonLeaves}
non-leaf vertices $t\in T_D$, $\sons(t) = \{s_1,s_2\}$, with $r_D = 1$ at the root, which fulfill for all $k = 1,\ldots ,r_t,$
\begin{equation}\label{equation:htNestedness}
U^{(t)}[\bullet,k] = \sum_{k_1 =1}^{r_{s_1}}\sum_{k_2 =1}^{r_{s_2}} \mathbf b^{(t)}[k,k_1,k_2] 
U^{(s_1)}[\bullet,k_1]\otimes U^{(s_2)}[\bullet,k_2],
\end{equation}
i.e. the transfer tensor $\mathbf b^{(t)}$ holds the coefficients of the linear combination
\eqref{equation:htNestedness} by which the frame $U^{(t)}$ can be obtained from the frames 
$U^{(s_1)}$ and $U^{(s_2)}$. 
\end{enumerate}
The vector $\mathbf r = (r_t)_{t\in T_D}$ containing the numbers of columns of the frames
$U^{(t)}$, $t\in T_D$, is called representation rank of the HT-format. Notice that we always
assume $r_D = 1$ at the root, i.e. $U^{(D)}$ contains only one column, namely 
$\mathcal{M}_D(\mathbf a) = \mathbf a$, cf. \eqref{equation:rootMatricization}.
\end{definition}

\begin{definition}[Hierarchical Tucker rank (HT-rank)]
The Hierarchical Tucker rank (HT-rank) of a tensor $\mathbf a\in\mathbb{R}^I$, 
$I = \bigtimes_{\mu = 1}^d I_\mu$, w.r.t. some dimension tree $T_D$,
$D = \{1,\ldots ,d\}$, is defined as the vector $\mathbf r = (r_t)_{t\in T_D}$ containing the
smallest possible numbers $r_t$, $t\in T_D$, such that an HT-format of $\mathbf a$ with 
representation rank $\mathbf r$ exists. 
\end{definition}

Notice that by \eqref{equation:tMatricizationDecomposition} the components $r_t$, $t\in T_D$, of the 
HT-rank of $\mathbf a\in\mathbb{R}^I$, $I = \bigtimes_{\mu = 1}^d I_\mu$, w.r.t to some dimension
tree $T_D$ are the ranks of the respective matricizations
$\mathcal{M}_t(\mathbf a)\in\mathbb{R}^{I_t\times I_{[t]}}$, denoted as
\begin{equation*}
\rank_t(\mathbf a) := \rank(\mathcal{M}_t(\mathbf a)).
\end{equation*}
According to \eqref{equation:setRr} 
for $\mathbf r\in\mathbb{N}_0^{T_D}$,
$r_D = 1$, we define the set $\mathcal{H}_{\mathbf r}$ as
\begin{equation}\label{equation:setHr}
\mathcal{H}_{\mathbf r} := \{\mathbf a\in\mathbb{R}^I\; :\; \rank_t(\mathbf a)\leq r_t,
\; t\in T_D\},
\end{equation}
i.e. the set of all tensors for which an HT-representation with representation rank 
$\mathbf r$ exists. It can be shown that $\mathcal{H}_{\mathbf r}$ is a closed subset
of $\mathbb R^I$, cf. \cite{Ha12}.

\paragraph*{Evaluation of tensor entries}
Let $\mathbf a\in\mathbb{R}^I$, $I = \bigtimes_{\mu = 1}^d I_\mu$, be represented in the HT-format
w.r.t. some dimension tree $T_D$, i.e. at the leaf vertices $\{\mu\}\in T_D$ the matrices
$U^{(\{\mu\})}\in\mathbb{R}^{I_\mu\times r_{\{\mu\}}}$ are stored and at the non-leaf vertices
$t\in T_D$ the transfer tensors $\mathbf b^{(t)}\in\mathbb{R}^{r_t\times r_{s_1}\times r_{s_1}}$,
$\sons(t) = \{s_1,s_2\}$, are stored. The tensor entries
$\mathbf a[\mathbf i]$, $\mathbf i = (i_1,\ldots ,i_d)\in I$, are thus not available explicitly
but can be computed recursively: by \eqref{equation:rootMatricization} we have
\begin{equation*}
\mathbf a[\mathbf i] = U^{(D)}[\mathbf i],
\end{equation*}
which can be evaluated by using \eqref{equation:htNestedness}:
\begin{equation}\label{equation:htRecursionRoot}
U^{(D)}[\mathbf i]
\underset{\eqref{equation:htNestedness}}{=} \sum_{k_1 =1}^{r_{t_1}}\sum_{k_2 =1}^{r_{t_2}}
\mathbf b^{(D)}[k_1,k_2] U^{(t_1)}[\mathbf i_{t_1},k_1] U^{(t_2)}[\mathbf i_{t_2},k_2],
\quad \sons(D) = \{t_1, t_2\},
\end{equation}
where we use the notation \eqref{equation:shortNotationIndices} for sub-indices
of $\mathbf i\in I$ corresponding to some subset $t\subseteq\{1,\ldots ,d\}$.
The columns $U^{(t_1)}[\bullet,k_1]$ and $U^{(t_2)}[\bullet,k_2]$ 
are either stored explicitly (if the corresponding vertex
$t_1$ or $t_2$ is a leaf of $T_D$) or can again be evaluated recursively by using
\eqref{equation:htNestedness}.
By this means, the tensor entries $\mathbf a[\mathbf i]$, $\mathbf i\in I$, of $\mathbf a$
can be evaluated with a computational work growing like $\mathcal{O}(dr^3)$, i.e. linearly with
the tensor order $d$, when $r\in\mathbb{N}$ is
an upper bound for the components of the representation rank $\mathbf r = (r_t)_{t\in T_D}$.

\paragraph*{Storage complexity}
The storage complexity for a tensor $\mathbf a\in\mathbb{R}^I$, $I = \bigtimes_{\mu = 1}^d I_\mu$,
in HT-representation w.r.t. some dimension tree $T_D$ also grows linearly with the tensor order
$d$ (as for the CP-format): assuming all tensor directions to be of the same size
$n = \#I_\mu$ and a bound $r\in\mathbb{N}$ for the components of the representation rank
$\mathbf r = (r_t)_{t\in T_D}$ of $\mathbf a$, $nr$ real numbers have to be stored at each
leaf vertex $\{\mu\}\in T_D$, $\mu = 1,\ldots ,d$, cf. 
Definition~\ref{definition:htFormat} \eqref{enumerate:htFormatLeaves}, 
and $r^3$ real numbers have to be stored at each interior 
vertex (neither leaf nor root) and
$r^2$ real numbers at the root. Taking into account that a dimension tree $T_D$, 
$D = \{1,\ldots ,d\}$, consists of $d$ leaf vertices and $d-1$ non-leaf vertices, this sums up
to a storage complexity of $\mathcal{O}(d(nr + r^3))$.

\paragraph*{Arithmetic operations in the HT-format}
Two arithmetic operations for tensors, which are of particular importance for this article,
are the dot product $\langle \mathbf x,\mathbf y\rangle\in\mathbb R$ and the Hadamard product 
$\mathbf x\circ\mathbf y\in\mathbb R^I$ of two tensors $\mathbf x, \mathbf y\in\mathbb{R}^I$, 
defined by
\[
\langle \mathbf x,\mathbf y\rangle  := \sum_{\mathbf i\in I}\mathbf x[\mathbf i]\cdot \mathbf y[\mathbf i],
\qquad
(\mathbf x\circ \mathbf y)[\mathbf i] :=
\mathbf x[\mathbf i]\mathbf \cdot \mathbf y[\mathbf i],\quad \mathbf i\in I.
\]
These operations can be computed directly in the HT-format, cf. \cite{Ha12,GrLoe18}.
The corresponding complexities grow again linearly with the tensor order $d$:
assuming all tensor directions to be of the same size $n\in\mathbb{N}$ and an upper bound 
$r\in\mathbb{N}$ for the components of the representation ranks of $\mathbf x$ and $\mathbf y$, 
dot products can be computed with a complexity bounded by
$\mathcal{O}(d(nr^2 + r^4))$ flops 
whereas 
$\mathcal{O}(d(nr^2 + r^6))$ flops 
are needed for the computation of Hadamard products (cf. \cite{GrLoe18} for a parallel algorithm).

A more detailed description of different low-rank formats for tensors can e.g. be found in
\cite{Ha12,GrKreTob13}. For this article we use the HT-format but other formats can be used
as well, e.g. the TT-format \cite{Osel11,Ha12} or MPS format \cite{VerCir06,Vidal03}.
Since the computation of the Hadamard product of two low-rank tensors typically yields a
result of higher representation rank, one needs a
truncation procedure which truncates tensors back to lower ranks. For the HT-format we 
use the truncation techniques introduced in \cite{Gr10} the complexity of which is comparable to
that of the computation of dot products above.

\section{Maximum entries of elementary tensors}\label{section:rank1}
Let $\mathbf a\in\mathbb{R}^I$, $I = \bigtimes_{\mu = 1}^d I_\mu$.
\begin{definition}[Maximum norm of a tensor]\label{definition:maximumNorm}
The maximum norm of a tensor $\mathbf a$
is defined by 
$\|\mathbf a\|_\infty := \max\{ |\mathbf a[\mathbf i]|\; :\; \mathbf i\in I\}$.
\end{definition}

\paragraph*{Maximum norm of a matrix}
Notice that Definition~\ref{definition:maximumNorm} may lead to an ambiguity:
for a matrix $M\in\mathbb{R}^{m\times n}$, $\|M\|_\infty$ is often used to denote the operator norm 
$\|M\|_{\mathrm{op},\infty}$ of $M$ as a linear operator 
$(\mathbb R^n,\|\cdot\|_\infty)\to (\mathbb R^m,\|\cdot\|_\infty)$ which does not match 
Definition~\ref{definition:maximumNorm}:
\begin{equation*}
\|M\|_{\mathrm{op},\infty} := \sup_{\|x\|_\infty = 1 }\|Mx\|_\infty 
 = \max_{i=1,\ldots ,m} \|M[i,\bullet]\|_1,
\end{equation*}
cf. \cite{Golub96}, where $M[i,\bullet]$ denotes the $i$-th row of $M$.
In the following $\|M\|_\infty$ is always used according to Definition~\ref{definition:maximumNorm},
i.e. $\|M\|_\infty = \max\{|M[i,j]|\; :\; i=1,\ldots ,m,\; j=1,\ldots ,n\}$.

\paragraph*{Maximum norm of an elementary tensor}
Let 
$\mathbf a = u^{(1)}\otimes\cdots\otimes u^{(d)}$, $u^{(\mu)}\in\mathbb{R}^{I_\mu}$,
$\mu = 1,\ldots ,d$, be an elementary tensor.
The maximum norm  of $\mathbf a$ is then given by
\begin{equation*}
\|\mathbf{a}\|_{\infty} = \max_{\mathbf{i}\in I}|a[\mathbf i]| = 
\prod_{\mu = 1}^d |u^{(\mu)}[\bullet]|_{\infty},
\end{equation*}
which can be maximized by maximizing each of the factors $|u^{(\mu)}[\bullet]|$, 
$\mu = 1,\ldots ,d$, in complexity ${\cal O}(dn)$.
Unfortunately, no such algorithm is known for tensors of higher rank (i.e. non-elementary tensors).
Even if a good approximation of a matrix $M$ by a rank-$1$ matrix $\tilde M$ exists in the 
Euclidean 
sense (e.g. a truncated singular value decomposition), the corresponding
maximum norms $\|M\|_\infty$ and $\|\tilde M\|_\infty$ may differ extremely as the following
example indicates:

\paragraph*{Example}
Consider a matrix $M\in\mathbb{R}^{n\times n}$ defined as
\begin{equation}\label{equation:counterexample}
M := \underbrace{\frac{\sigma_1}{n-1}
\left(
\begin{array}{r|rrr}
0 & & & \\ \hline
 & 1 & \cdots & 1 \\
 & \vdots & \ddots & \vdots \\
 & 1 & \cdots & 1
\end{array}
\right)}_{M_1} + 
\underbrace{\sigma_2 
\left(
\begin{array}{r|rrr}
1 & & & \\ \hline
 & 0 & \cdots & 0 \\
 & \vdots & \ddots & \vdots \\
 & 0 & \cdots & 0
\end{array}
\right)}_{M_2},\qquad \sigma_1 > \sigma_2 \geq 0.
\end{equation}
Then $M$ has the singular values $\sigma_1 > \sigma_2$ and $M_1$ is the rank-$1$ best approximation
of $M$ in the Frobenius (and spectral) norm
with an approximation error 
\begin{equation}\label{equation:approximationErrorFrobenius}
\frac{\| M - M_1\|_F}{\|M\|_F} = \sqrt{\frac{\sigma_2^2}{\sigma_1^2 + \sigma_2^2}} \leq
\frac{\sigma_2}{\sigma_1}
\end{equation}
in the Frobenius norm.
Assuming $\sigma_2 > \frac{\sigma_1}{n-1}$ we get a corresponding error
\begin{equation}\label{equation:approximationErrorAbsMax}
\frac{|\|M\|_\infty - \|M_1\|_\infty|}{\|M\|_\infty}
 = \frac{\sigma_2 - \frac{\sigma_1}{n-1}}{\sigma_2} = 1 - 
 \frac{\sigma_1}{\sigma_2}\cdot\frac{1}{n-1}
\end{equation}
for the approximation of the maximum norm, which tends to $1$ for $n\to\infty$, 
independently of the error \eqref{equation:approximationErrorFrobenius}.

\section{Estimating the maximum norm by a power iteration}\label{section:power}

We use a power iteration to approximate the maximum norm $\|\mathbf a\|_\infty$ of
a tensor $\mathbf a\in\mathbb{R}^I$, $I = \bigtimes_{\mu = 1}^d I_\mu$, as introduced in 
\cite{Espig08}:
consider the diagonal matrix 
\begin{equation}\label{equation:tensorAsDiagMatrix}
D\in\mathbb R^{I\times I},\qquad \diag(D) = \mathbf a.
\end{equation}
Then the maximum norm $\|\mathbf a\|_\infty$ is
the maximum absolute value of an eigenvalue of $D$:
\begin{equation}\label{equation:maximumNormAsEigenvalue}
\|\mathbf a\|_\infty = \max\{ |\lambda|\; :\; D\mathbf v = \lambda \mathbf v\;\mbox{for some}\;\mathbf v\neq \mathbf 0\},
\end{equation}
which can be computed by a standard power iteration if there exists only one eigenvalue 
$\lambda$ of $D$
with $|\lambda| = \|\mathbf a\|_\infty$, i.e. there is only one index $\mathbf i\in I$ with
$\|\mathbf a\|_\infty = |\mathbf a[\mathbf i]|$. Taking into account that the matrix multiplication
$D\mathbf v$, $\mathbf v\in\mathbb{R}^I$, can be written as the Hadamard product
$\mathbf a\circ\mathbf v$, this results in Algorithm~\ref{algorithm:powerIteration}, where 
$\|\mathbf a\|=\|\mathbf a\|_2$ denotes the Euclidean norm of $\mathbf a$.
\begin{algorithm}
$\mathbf a^{(1)} := \mathbf a / \|\mathbf a\|$\;
\For { $j = 1,2,\ldots$ }
{
$\mathbf a^{(j+1)} := \mathbf a\circ \mathbf a^{(j)}$\;\label{line:hadamardProductStable}
$\lambda^{(j+1)} := \langle \mathbf a^{(j)},\mathbf a^{(j+1)}\rangle$ \tcp{Rayleigh quotient}\label{line:rayleighQuotient}
$\mathbf a^{(j+1)} \leftarrow \mathbf a^{(j+1)}/\|\mathbf a^{(j+1)}\|$\;
}
\caption{Power iteration: the maximum norm $\|\mathbf a\|_\infty$ of a tensor $\mathbf a\in\mathbb{R}^I$, $I = \bigtimes_{\mu = 1}^d I_\mu$, is approximated by  
$|\lambda^{(j+1)}|$, $j\in\mathbb{N}$, if only one index $\mathbf i\in I$ with 
$\|\mathbf a\|_\infty = |\mathbf a[\mathbf i]|$ exists.}
\label{algorithm:powerIteration}
\end{algorithm}

\paragraph*{Improvement of the Rayleigh quotient}
Due to the diagonal structure of the eigenvalue problem in 
\eqref{equation:maximumNormAsEigenvalue}, the Rayleigh quotient 
in line~\ref{line:rayleighQuotient} of Algorithm~\ref{algorithm:powerIteration} can be improved, 
which results in
a new estimator $\alpha^{(j+1)} :=\|\mathbf a^{(j+1)}\|$, $j\in\mathbb{N}$, cf. 
Algorithm~\ref{algorithm:powerIterationImproved}.
\begin{lemma}[Improved estimator for the power iteration]\label{lemma:improvedEstimator}
The new estimator $\alpha^{(j+1)}$, $j\in\mathbb{N}$, of
Algorithm~\ref{algorithm:powerIterationImproved} is at least as good as the Rayleigh quotient 
$\lambda^{(j+1)}$, $j\in\mathbb{N}$, from the standard power iteration, cf. 
Algorithm~\ref{algorithm:powerIteration}:
\begin{equation}\label{equation:rayleighImprovement}
|\lambda^{(j+1)}|\leq \alpha^{(j+1)},\qquad j \in\mathbb N,
\end{equation}
and always converges to $\|\mathbf a\|_\infty$ from below: $\alpha^{(j+1)}\uparrow\|\mathbf a\|_\infty$,
$j\to\infty$.
\end{lemma}
\begin{proof}
Using the Cauchy-Schwarz inequality and $\|\mathbf a^{(j)}\| = 1$ in 
line~\ref{line:rayleighQuotient} of Algorithm~\ref{algorithm:powerIteration} yields
\begin{equation*}
|\lambda^{(j+1)}| = |\langle \mathbf a^{(j)},\mathbf a^{(j+1)}\rangle|
\leq \|\mathbf a^{(j)}\| \|\mathbf a^{(j+1)}\| = \|\mathbf a^{(j+1)}\| = \alpha^{(j+1)},
\end{equation*}
i.e. \eqref{equation:rayleighImprovement}. 
By induction one easily verifies that 
$\mathbf a^{(j)} = \mathbf a^{\circ j}/\|\mathbf a^{\circ j}\|$, $j\in\mathbb{N}$, 
holds in line~\ref{line:improvedRayleighQuotient} of Algorithm~\ref{algorithm:powerIterationImproved}, where
$\mathbf a^{\circ j}$ denotes the $j$-fold Hadamard product of $\mathbf a$. 
This implies
\begin{equation}\label{equation:alphajLowerBound}
\alpha^{(j+1)} = \|\mathbf a\circ \mathbf a^{(j)}\| 
= \frac{\|\mathbf a\circ \mathbf a^{\circ j}\|}{\|\mathbf a^{\circ j}\|} \leq
\frac{\|\mathbf a\|_\infty \|\mathbf a^{\circ j}\|}{\|\mathbf a^{\circ j}\|} = \|\mathbf a\|_\infty,
\end{equation}
in line~\ref{line:improvedRayleighQuotient}, i.e. $\alpha^{(j+1)}$, $j\in\mathbb{N}$, are lower
bounds of $\|\mathbf a\|_\infty$.
We now use
\begin{equation}\label{equation:powerNorm}
\|\mathbf a^{\circ j}\| = \left(\sum_{\mathbf i\in I}\mathbf a[\mathbf i]^{2j}\right)^{1/2}
 = \left(\left(\sum_{\mathbf i\in I}\mathbf a[\mathbf i]^{2j}\right)^{1/2j}\right)^j
  = \|\mathbf a \|_{2j}^j
\end{equation}
and the norm equivalence for $p$-norms on $\mathbb R^I$ which follows from the H\"older inequality:
\begin{equation}\label{equation:normEquivalence}
\|\mathbf a\|_{2j} \leq N^{\frac{1}{2j} - \frac{1}{2(j+1)}} \|\mathbf a\|_{2(j+1)}
 = N^{\frac{1}{2j(j+1)}} \|\mathbf a\|_{2(j+1)},
\qquad N = \# I,
\end{equation}
as well as 
\begin{equation}\label{equation:inequalityMaxNormPNorm}
\|\mathbf a\|_\infty \leq \|\mathbf a\|_{2(j+1)},
\end{equation}
in order to get a lower bound for $\alpha^{(j+1)}$, $j\in \mathbb{N}$:
\begin{align}
\alpha^{(j+1)} & \underset{\eqref{equation:alphajLowerBound}}{=} \notag
\frac{\|\mathbf a\circ \mathbf a^{\circ j}\|}{\|\mathbf a^{\circ j}\|}
 = \frac{\|\mathbf a^{\circ(j+1)}\|}{\|\mathbf a^{\circ j}\|}
 \underset{\eqref{equation:powerNorm}}{=}
 \frac{\|\mathbf a\|_{2(j+1)}^{j+1}}{\|\mathbf a\|_{2j}^j}
 \underset{\eqref{equation:normEquivalence}}{\geq} \frac{\|\mathbf a\|_{2(j+1)}^{j+1}}{N^{1/(2(j+1))} \|\mathbf a\|_{2(j+1)}^j}\\
 & = N^{-\frac{1}{2(j+1)}} \|\mathbf a\|_{2(j+1)}
 \underset{\eqref{equation:inequalityMaxNormPNorm}}{\geq}
 N^{-\frac{1}{2(j+1)}} \|\mathbf a\|_\infty. \label{equation:alphajUpperBound}
\end{align}
Combining the bounds \eqref{equation:alphajLowerBound} and \eqref{equation:alphajUpperBound} for
$\alpha^{(j+1)}$ yields
\begin{equation}\label{equation:convergenceEstimation}
0\leq \|\mathbf a\|_\infty - \alpha^{(j+1)} \leq (1-N^{-\frac{1}{2(j+1)}})\|\mathbf a\|_\infty,
\end{equation}
i.e. $\alpha^{(j+1)}\uparrow \|\mathbf a\|_\infty$ for $j\to\infty$ since
$1-N^{-\frac{1}{2(j+1)}}\to 0$, $j\to\infty$.
\end{proof}
\begin{algorithm}
$\mathbf a^{(1)} := \mathbf a / \|\mathbf a\|$\;
\For { $j = 1,2,\ldots$ }
{
$\mathbf a^{(j+1)} := \mathbf a\circ \mathbf a^{(j)}$\; \label{line:newIterateStandard}
$\alpha^{(j+1)} := \|\mathbf a^{(j+1)}\|$ \tcp{improved estimator}\label{line:improvedRayleighQuotient}
$\mathbf a^{(j+1)} \leftarrow \mathbf a^{(j+1)}/\|\mathbf a^{(j+1)}\|$\;
}
\caption{Power iteration: $\alpha^{(j+1)}\uparrow \|\mathbf a\|_\infty$ for $j\to\infty$, where $\alpha^{(j+1)}$ is bounded from below by the absolute value $|\lambda^{(j+1)}|$ of the Rayleigh quotient from Algorithm~\ref{algorithm:powerIteration}, cf. Lemma~\ref{lemma:improvedEstimator}.}
\label{algorithm:powerIterationImproved}
\end{algorithm}

The following example shows that the estimator $\alpha^{(j+1)}$ from 
Algorithm~\ref{algorithm:powerIterationImproved} is indeed better than the Rayleigh quotient from
the standard power iteration:

\paragraph*{Example}
For the vector (order-$1$ tensor) $\mathbf a = [1,-1]\in\mathbb{R}^2$ 
the standard power iteration from Algorithm~\ref{algorithm:powerIteration} does not converge to
$\|\mathbf a\|_\infty$:
\begin{equation*}
\lambda^{(j+1)} = 0,\qquad j\in\mathbb{N}.
\end{equation*}
The estimator $\alpha^{(j+1)}$ from Algorithm~\ref{algorithm:powerIterationImproved}, 
however, immediately converges to $\|\mathbf a\|_\infty$:
\begin{equation*}
\alpha^{(j+1)} = 1 = \|\mathbf a\|_\infty,\qquad  j\in\mathbb{N}.
\end{equation*}

\paragraph*{Error estimator for Algorithm~\ref{algorithm:powerIterationImproved}}
Since Algorithm~\ref{algorithm:powerIterationImproved} 
is a power iteration for $\mathbf a\in\mathbb R^I$ with an improved
estimator for the eigenvalue, cf. Lemma~\ref{lemma:improvedEstimator}, any error 
estimator known for the standard power iteration can 
also be applied to Algorithm~\ref{algorithm:powerIterationImproved}. 
Since in our framework the underlying matrix is diagonal, cf. \eqref{equation:tensorAsDiagMatrix}, 
results for power iteration with symmetric matrices may be applied.
E.g. in \cite{Kucz92} the error ${\varepsilon^{(j+1)} :=(\|\mathbf a\|_\infty - \alpha^{(j+1)}) / \|\mathbf a\|_\infty}$, $j\in\mathbb N$,
is proven to be bounded by 
\begin{equation}\label{equation:errorPowerIterationSymmetric}
\frac{\ln(N)}{j-1}, \quad j\geq 2,\quad N = \# I, 
\end{equation}
if the corresponding matrix
is symmetric positive definite, which in our framework means that all entries of $\mathbf a$ 
are positive, cf. \eqref{equation:tensorAsDiagMatrix}.

However, for our special case of a power iteration on a diagonal matrix $D$, the bound
\begin{equation}\label{equation:errorPowerIterationDiagonal}
\varepsilon^{(j+1)}\leq 1-N^{-\frac{1}{2(j+1)}}
\end{equation}
can directly be obtained from \eqref{equation:convergenceEstimation} in the proof of 
Lemma~\ref{lemma:improvedEstimator}. Asymptotically the bounds 
\eqref{equation:errorPowerIterationSymmetric} and \eqref{equation:errorPowerIterationDiagonal} 
show the same behavior, where the bound in \eqref{equation:errorPowerIterationDiagonal} 
is always better than the one in \eqref{equation:errorPowerIterationSymmetric}.
Assuming an error which grows like the bound \eqref{equation:errorPowerIterationSymmetric} leads
to a slow convergence:
\begin{equation}\label{equation:badConvergence}
\varepsilon^{(j+1)} \approx \left(1-\frac{1}{j-1}\right)\varepsilon^{(j)}.
\end{equation}

Even though the above error estimates may be pessimistic in some cases, 
our experiments show that the convergence behavior \eqref{equation:badConvergence} in fact
occurs for many standard examples, cf. Figures~\ref{figure:powerIterationSimple} and
\ref{subfigure:histogramPI}.

\paragraph*{Complexity of the power iteration}
In the HT-format all tensor operations involved in 
Algorithm~\ref{algorithm:powerIterationImproved} can be
carried out directly in the HT-format with a complexity depending linearly on the tensor order
$d$ and the mode sizes.
We can thus expect such a linear complexity for the approximation of $\|\mathbf a\|_\infty$ 
by Algorithm~\ref{algorithm:powerIterationImproved} 
with $\mathbf a\in\mathbb R^I$, $I = \bigtimes_{\mu = 1}^d I_\mu$, in HT-representation,
which is verified by the experiments in Section~\ref{subsection:experimentsArgmax},
cf. Figure~\ref{figure:argmaxBinarySearch}.

\paragraph*{Truncation}
Since the computation of Hadamard products increases the representation ranks of the iterates,
we include truncations in each step of Algorithm~\ref{algorithm:powerIterationImproved} 
in order to keep the 
representation ranks of the iterates small enough. 
We compute the truncations by a hierarchical singular value decomposition (HSVD) in the
HT-format \cite{Gr10}.
A detailed analysis of truncations in iterative fixed-point like processes
is given by \cite{Ha08}: if the truncation errors are small enough and the starting tensor is
chosen appropriately, the iterative process still
converges to the original fixed point, provided that the latter allows for a respective 
low-rank 
representation.
In our numerical experiments, cf. Section~\ref{section:experiments}, we truncate w.r.t. fixed
prescribed HT-ranks
in each step of the algorithms, which already worked well.
Another option would be to truncate w.r.t. a given tolerance $\varepsilon > 0$ for the 
relative truncation error, which, however, may lead to high representation ranks in the beginning
of the algorithms, since the underlying error estimations for the HT-format can be rather
pessimistic, cf. \cite{Gr10}.

Notice that it is possible to construct counterexamples where any truncation leads to the loss of
that entry of $\mathbf a\in\mathbb{R}^I$ which corresponds to $\|\mathbf a\|_\infty$, i.e. to a
wrong approximation of $\|\mathbf a\|_\infty$ in the end.
Such counterexamples can be constructed as in \eqref{equation:counterexample}: take $\mathbf a = M$ (order-$2$ tensor) as defined in
\eqref{equation:counterexample} with $\sigma_1 = (n-1)\varepsilon$ and
$\sigma_2 = 1$, $1/(n-1) < \varepsilon < 1$.
A rank-$1$ best approximation of $M$ is given by $M_1$, cf.
\eqref{equation:counterexample}, i.e. after truncation, instead of
$\| M\|_\infty = 1$, the maximum norm 
$\|M_1\|_\infty = \varepsilon$ would be computed.
This example shows that it may happen that a single entry of larger absolute value 
(the entry $1$ of $M$) is "hidden" by many entries of smaller absolute value 
(the entries $\varepsilon$ of $M$).
The reason for this is that the truncation, cf. \cite{Gr10}, is based on the singular value 
decomposition, 
which corresponds to a best approximation in the Euclidean sense and not w.r.t. the maximum norm.
Nevertheless this does not stand in contradiction to the convergence analysis developed 
in \cite{Ha08} since
the truncation of $M$ does not meet the underlying assumptions:
the starting tensor $\mathbf a/\|\mathbf a\| = M / \|M\|_F$ of the 
iterative process is far away from the fixed point $M_2$ 
($\|M_2\|_F = 1$), 
cf. \eqref{equation:counterexample}:
\begin{align*}
\left\|\frac{M}{\|M\|_F} - M_2\right\|_F &= 
\sqrt{\frac{\bigl(1-\sqrt{(n-1)^2\varepsilon^2 + 1}\bigr)^2 + (n-1)^2\varepsilon^2}{(n-1)^2\varepsilon^2 + 1}} \\
& \geq \sqrt{\frac{(n-1)^2\varepsilon^2}{1+(n-1)^2\varepsilon^2}} \to 1\;,
\qquad
n\to\infty.
\end{align*}

\section{Estimating the maximum norm by Ritz values}\label{section:rayritz}

In order to accelerate the convergence of the power iteration, we make use of an orthogonal 
projection technique
which uses not only the iterate $\mathbf a^{(j+1)}$ of the current step $j\in \mathbb N$
for the estimation of $\|\mathbf a\|_\infty$,
cf. line~\ref{line:improvedRayleighQuotient} in Algorithm~\ref{algorithm:powerIterationImproved}, 
but the latest $k$ iterates, i.e. $\mathbf a^{(j-k+2)},\ldots ,\mathbf a^{(j+1)}$, for some 
$k\in\mathbb{N}$, $k\leq j+1$.
The resulting method is often referred to as Rayleigh-Ritz method, cf. \cite{BaiDemmel2000}:

\paragraph*{Rayleigh-Ritz method}
The idea is to approximate an exact eigenvector $\mathbf v\in\mathbb R^I$ of the underlying
eigenvalue problem $D\mathbf v = \lambda\mathbf v$, $|\lambda| = \|\mathbf a\|_\infty$, 
cf. \eqref{equation:tensorAsDiagMatrix}, 
\eqref{equation:maximumNormAsEigenvalue}, by a linear combination of the latest $k$
iterates $\mathbf a^{(j-k+2)},\ldots ,\mathbf a^{(j+1)}$, $k\leq j+1$, 
which means that an approximation $\tilde{\mathbf v}$ of $\mathbf v$ is searched for in the 
corresponding subspace
\begin{equation}\label{equation:rayritzSubspace}
U_k :=\spn\{ \mathbf a^{(j-k+2)},\ldots ,\mathbf a^{(j+1)}\}.
\end{equation}
The approximation $\tilde{\mathbf v}$ and a corresponding approximation $\tilde{\lambda}$ of
$\lambda$ are obtained by imposing a Galerkin condition:
\begin{equation*}
D\tilde{\mathbf v} - \tilde\lambda \tilde{\mathbf v} \;\bot\; U_k.
\end{equation*}
Computing an orthonormal basis $\{\mathbf q^{(1)},\ldots ,\mathbf q^{(k)}\}$ of $U_k$
and defining the symmetric matrix $B_k\in\mathbb{R}^{k\times k}$ as 
\begin{equation}\label{equation:symmetricMatrixBk}
B_k := Q_k^\top D Q_k,\qquad Q_k := [\mathbf q^{(1)} \mid\cdots\mid \mathbf q^{(k)}]
\in\mathbb{R}^{I\times k},
\end{equation}
yields 
\begin{equation*}
B_k \cdot w = \tilde\lambda w,
\end{equation*}
i.e. $\tilde\lambda$ is the eigenvalue of $B_k$ which has maximum absolute value, and
$\tilde{\mathbf v} = Q_k w$.
The eigenvalues of $B_k$ are called Ritz values of $D$, i.e. the eigenvalue $\lambda$ of $D$
fulfilling $|\lambda| = \|\mathbf a\|_\infty$ is approximated by the Ritz value $\tilde\lambda$ 
of $D$ which maximizes $|\tilde \lambda|$.

\subsection{Computation of orthonormal bases in the HT-format}\label{subsection:HTQR}
Since we use the HT-format for the representation of low-rank tensors, the iterates
$\mathbf a^{(j-k+2)},\ldots ,\mathbf a^{(j+1)}$ spanning the subspace $U_k$ are tensors in the
HT-format. 
We thus want to compute an orthonormal basis $\{\mathbf q^{(1)},\ldots ,\mathbf q^{(k)}\}$ of $U_k$
directly in the HT-format. This can e.g. be achieved by a Gram-Schmidt process which only involves
the computation of sums and dot products and can thus be carried out directly in the HT-format.

For our experiments in Section~\ref{section:experiments} we used a different procedure which
computes a set $\{\mathbf q^{(1)},\ldots ,\mathbf q^{(k)}\}$ of orthonormal HT-tensors together with
a right factor $R\in\mathbb R^{k\times k}$, such that
\begin{equation}\label{equation:htQR}
[\mathbf a^{(j-k+2)}\mid\cdots \mid\mathbf a^{(j+1)}] = 
[\mathbf q^{(1)}\mid\cdots \mid\mathbf q^{(k)}]\cdot R.
\end{equation}
We now briefly describe this procedure which we refer to as HT-QR decomposition.
In \cite{Ha12} the computation of joint orthonormal bases (ONB) is described for a set of tensors 
in HT-representation: 
at each non-root vertex $t\in T_D$ of the underlying dimension tree $T_D$, a 
QR-decomposition of the matrix containing the corresponding frames of the tensors
$\mathbf a^{(j-k+2)},\ldots ,\mathbf a^{(j+1)}$ is computed:
\begin{equation}\label{equation:frameQR}
[U^{(t,j-k+2)}\mid\cdots\mid U^{(t,j+1)}] = 
Q^{(t)}\cdot [R^{(t,j-k+2)}\mid\cdots\mid R^{(t,j+1)}].
\end{equation}
The orthonormal columns of $Q^{(t)}$ are kept as new (joint, orthonormal) 
basis at the vertex $t$ for all
tensors $\mathbf a^{(j-k+2)},\ldots ,\mathbf a^{(j+1)}$ and the right factors
$R^{(t,j-k+2)},\ldots ,R^{(t,j+1)}$ are multiplied to the father vertices, where the 
procedure continues recursively. By this, joint ONB are installed at all
non-root vertices $t\in T_D\setminus\{D\}$, i.e. the new HT-representations 
of $\mathbf a^{(j-k+2)},\ldots ,\mathbf a^{(j+1)}$ only differ at the root vertex.
At the leaf vertices, a QR-decomposition \eqref{equation:frameQR} can be computed 
directly since the frame matrices $U^{(t,j-k+2)},\ldots ,U^{(t,j+1)}$ are stored 
explicitly. 
At the interior vertices, \eqref{equation:frameQR} can be 
computed by a respective QR-decomposition of
\begin{equation*}
\bigl[\mathcal{M}_{\{2,3\}}(\mathbf b^{(t,j-k+2)})\mid\cdots\mid \mathcal{M}_{\{2,3\}}(\mathbf b^{(t,j+1)})\bigr] \quad\mbox{(transfer tensors as columns)}.
\end{equation*}

Notice that the computation of joint ONB for the set 
$\{\mathbf a^{(j-k+2)},\ldots ,\mathbf a^{(j+1)}\}$ only affects their HT-representations
while the tensors themselves stay invariant. After the HT-representations of
$\{\mathbf a^{(j-k+2)},\ldots ,\mathbf a^{(j+1)}\}$ have been transformed into joint ONB, an
HT-QR decomposition \eqref{equation:htQR} can easily be obtained by computing a 
respective QR-decomposition of the
root transfer tensors $\mathbf b^{(D,j-k+2)},\ldots ,\mathbf b^{(D,j+1)}$. 

The resulting accelerated power iteration is shown in 
Algorithm~\ref{algorithm:powerIterationRay}: we start with 
Algorithm~\ref{algorithm:powerIterationImproved}, where we leave out the computation
of the estimator $\alpha^{(j+1)}$ in each step, which is now computed in the end
as Ritz value of maximum absolute value based on the last $k$ iterates.
\begin{algorithm}
$\mathbf a^{(1)} := \mathbf a / \|\mathbf a\|$\;
\For { $j = 1,2,\ldots$ }
{
$\mathbf a^{(j+1)} := \mathbf a\circ \mathbf a^{(j)}$\; 
$\mathbf a^{(j+1)} \leftarrow \mathbf a^{(j+1)}/\|\mathbf a^{(j+1)}\|$\;
}
$[\mathbf q^{(1)}\mid\cdots\mid \mathbf q^{(k)}] = \mbox{HT-QR}([\mathbf a^{(j-k+2)}\mid\cdots\mid \mathbf a^{(j+1)}])$\tcp*{Compute an ONB}
\For { $i_1 = 1,2,\ldots ,k$}
{
\For { $i_2 = 1,2,\ldots ,k$}
{
$(B_k)[i_1,i_2] = \langle \mathbf q^{(i_1)},\mathbf a\circ \mathbf q^{(i_2)}\rangle$
\tcp*{Matrix $B_k$ from \eqref{equation:symmetricMatrixBk}}
}
}
Compute the eigenvalues $\lambda_\ell$, $\ell = 1,\ldots ,k$, of $B_k$\;
$\alpha^{(j+1)} := \max_{\ell = 1,\ldots ,k}|\lambda_\ell|$ \tcp{Estimator for $\|\mathbf a\|_\infty$}
\caption{Acceleration of the power iteration using a Rayleigh-Ritz type orthogonal projection method: the estimator $\alpha^{(j+1)}$ now depends on the $k$ latest iterates of the power iteration, $1\leq k\leq j+1$.}
\label{algorithm:powerIterationRay}
\end{algorithm}

\section{Acceleration by squaring}\label{section:squaring}

Even though the Rayleigh-Ritz method (Algorithm~\ref{algorithm:powerIterationRay}), 
cf. Section~\ref{section:rayritz}, 
may lead to a notable acceleration of the power iteration 
(Algorithm~\ref{algorithm:powerIterationImproved}), 
cf. Figure~\ref{subfigure:powerIterationProjectionRand} and 
Figure~\ref{figure:histogramPIPR},
there are also cases where the order of convergence stays invariant,
cf. Figure~\ref{subfigure:powerIterationProjectionCheb}.
However, the power iteration including a Rayleigh-Ritz method can well be used to generate 
sufficiently good starting tensors for a faster (but less stable) method, which is presented
in this section.

\paragraph*{Power iteration with squaring}
As stated in the proof of Lemma~\ref{lemma:improvedEstimator}, the $j$-th iterate 
$\mathbf a^{(j)}$ of Algorithm~\ref{algorithm:powerIterationImproved}
is the corresponding normalized power of the original
tensor $\mathbf a\in\mathbb{R}^I$ (if truncations down to smaller HT-representation rank
are neglected):
\begin{equation}\label{equation:iteratePowerIteration}
\mathbf a^{(j)} = \frac{\mathbf a^{\circ j}}{\|\mathbf a^{\circ j}\|},\qquad j\in\mathbb N.
\end{equation}
The power iteration can be sped up by taking the new iterate 
\begin{equation*}
\mathbf a^{(j+1)} := \frac{\mathbf a^{(j)}\circ \mathbf a^{(j)}}{\|\mathbf a^{(j)}\circ \mathbf a^{(j)}\|}
\end{equation*}
instead of line~\ref{line:newIterateStandard} in Algorithm~\ref{algorithm:powerIterationImproved},
which results in Algorithm~\ref{algorithm:powerIterationSquaring}, hereinafter referred to as
power iteration with squaring.
\begin{algorithm}
$\mathbf a^{(1)} := \mathbf a / \|\mathbf a\|$\;
\For { $j = 1,2,\ldots$ }
{
$\mathbf a^{(j+1)} := (\mathbf a^{(j)}\circ\mathbf a^{(j)})/ \|\mathbf a^{(j)}\circ \mathbf a^{(j)}\|$\;\label{line:hadamardProductLessStable}
$\alpha^{(j+1)} := \|\mathbf a\circ \mathbf a^{(j+1)}\|$\;\label{line:alphaSquaring}
}
\caption{Power iteration with squaring: $\alpha^{(j+1)}\uparrow \|\mathbf a\|_\infty$ for 
$j\to\infty$.}
\label{algorithm:powerIterationSquaring}
\end{algorithm}
Without truncations, the iterates of Algorithm~\ref{algorithm:powerIterationSquaring} are also
normalized powers of the original tensor $\mathbf a$:
\begin{equation}\label{equation:iteratePowerIterationSquaring}
\mathbf a^{(j)} = \frac{\mathbf a^{\circ 2^{j-1}}}{\|\mathbf a^{\circ 2^{j-1}}\|},
\end{equation}
which yields the corresponding new expression for $\alpha^{(j+1)}$ in
line~\ref{line:alphaSquaring} of Algorithm~\ref{algorithm:powerIterationSquaring}.

The comparison of \eqref{equation:iteratePowerIteration} and \eqref{equation:iteratePowerIterationSquaring}
shows the immense acceleration which can be expected by using
Algorithm~\ref{algorithm:powerIterationSquaring} instead of 
Algorithm~\ref{algorithm:powerIterationImproved}: the iterate in step $j$ of 
Algorithm~\ref{algorithm:powerIterationSquaring} corresponds to the iterate in step $2^{j-1}$ of
Algorithm~\ref{algorithm:powerIterationImproved}. Assuming the slow convergence
\eqref{equation:badConvergence} for Algorithm~\ref{algorithm:powerIterationImproved}, 
cf. Figures \ref{figure:powerIterationSimple} and \ref{subfigure:histogramPI},
the corresponding convergence behavior of Algorithm~\ref{algorithm:powerIterationSquaring} would be
\begin{equation}\label{equation:linearConvergenceSquaring}
\varepsilon^{(j+1)} \approx \frac{2^{j-1} -2}{2^j - 2}\varepsilon^{(j)}
\leq \frac{1}{2}\varepsilon^{(j)},
\end{equation}
i.e. linear convergence with a convergence rate of $0.5$ or better, which is verified by 
our experiments in Section~\ref{subsection:powerIterationSquaring}, 
cf. Figure~\ref{figure:histogramSQAD}.

\paragraph*{Truncation}
As in Algorithm~\ref{algorithm:powerIterationImproved}, in practice truncations are included into
Algorithm~\ref{algorithm:powerIterationSquaring}, since otherwise the complexity of the involved
tensor arithmetic quickly becomes too large, cf. Section~\ref{subsection:htFormat}.
Algorithm~\ref{algorithm:powerIterationSquaring}, 
however, is less stable than Algorithm~\ref{algorithm:powerIterationImproved},
which can be explained as follows:
consider two tensors $\mathbf x,\, \mathbf y\in\mathbb{R}^I$, $I = \bigtimes_{\mu = 1}^d I_\mu$,
the Hadamard product of which is to be computed. If only $\mathbf x$ is perturbed by
a relative error (due to truncation), 
the computed Hadamard product will have the same relative error 
(notice that $\|\tilde{\mathbf x} - \mathbf x\|_p / \|\mathbf x\|_p\leq \|\mathbf \epsilon\|_\infty$ for any $p$-norm, $1\leq p\leq\infty$, if $\tilde{\mathbf x} = \mathbf x\circ(\mathbf 1\circ \mathbf\epsilon)$):
\begin{equation}\label{equation:hadamardProductStable}
\tilde{\mathbf x} = \mathbf x\circ (\mathbf 1 + \mathbf\epsilon)
 \quad\Rightarrow\quad \tilde{\mathbf x}\circ \mathbf y = (\mathbf x\circ \mathbf y)\circ (\mathbf 1 + \mathbf\epsilon),
\end{equation}
whereas the relative error doubles if $\mathbf x$ and $\mathbf y$ are both afflicted with the relative error:
\begin{equation}\label{equation:hadamardProductLessStable}
\tilde{\mathbf x} = \mathbf x\circ(\mathbf 1 + \mathbf \epsilon),\quad
\tilde{\mathbf y} = \mathbf y\circ(\mathbf 1 + \mathbf \epsilon)\quad\Rightarrow\quad
\tilde{\mathbf x}\circ\tilde{\mathbf y} = 
(\mathbf x \circ \mathbf y)\circ (\mathbf 1 + 2\mathbf \epsilon + \mathbf \epsilon\circ\mathbf \epsilon).
\end{equation}
In line~\ref{line:newIterateStandard} of Algorithm~\ref{algorithm:powerIterationImproved} the
Hadamard product $\mathbf a\circ \mathbf a^{(j)}$ is computed which corresponds to the case
\eqref{equation:hadamardProductStable}, if we assume the original tensor $\mathbf a$ to be
exact. In line~\ref{line:hadamardProductLessStable} of 
Algorithm~\ref{algorithm:powerIterationSquaring}, instead, the Hadamard product
$\mathbf a^{(j)}\circ\mathbf a^{(j)}$ of two truncated tensors is computed, which corresponds to
\eqref{equation:hadamardProductLessStable} with $\mathbf x = \mathbf y$.
Similar arguments apply for the computation of the norms in
Algorithms \ref{algorithm:powerIterationImproved} and \ref{algorithm:powerIterationSquaring}. 

Using the power iteration with squaring, the truncation error of the iterates 
$\mathbf a^{(j)}$, $j\in\mathbb{N}$, thus has to be small enough in 
order for the algorithm to converge to a good approximation of $\|\mathbf a\|_\infty$.
We thus use some steps of Algorithm~\ref{algorithm:powerIterationRay} to generate a
sufficiently good starting tensor for Algorithm~\ref{algorithm:powerIterationSquaring}.

\paragraph*{Algorithm for the estimation of $\|\mathbf a\|_\infty$}
Notice that the iterates $\mathbf a^{(j)}$, $j\in \mathbb N$, in the power iteration eventually
approximate corresponding eigenvectors of the matrix $D$, cf. \eqref{equation:tensorAsDiagMatrix}, 
\eqref{equation:maximumNormAsEigenvalue}, which usually means that they have only few non-zero
entries and thus small HT-ranks, i.e. the truncation error is small for reasonable
representation ranks of the HT-formats. Our strategy is therefore to start with some
steps of Algorithm~\ref{algorithm:powerIterationRay}
in order to get a good starting tensor for 
Algorithm~\ref{algorithm:powerIterationSquaring}, which we sketched out in
Algorithm~\ref{algorithm:adaptive}.
The number of steps of Algorithm~\ref{algorithm:powerIterationRay} is chosen 
adaptively: after $N_{\ref{algorithm:powerIterationRay}}$ steps of 
Algorithm~\ref{algorithm:powerIterationRay} we run
Algorithm~\ref{algorithm:powerIterationSquaring} until the iterates differ only by less
than some prescribed relative accuracy (in our experiments we chose 
$\|\mathbf a^{(j+1)} -\mathbf a^{(j)}\| / \|\mathbf a^{(j+1)}\| < 10^{-13}$ as a 
stopping criterion). If the truncation error becomes too 
large during the execution of Algorithm~\ref{algorithm:powerIterationSquaring}, 
another $N_{\ref{algorithm:powerIterationRay}}$ steps of
Algorithm~\ref{algorithm:powerIterationRay} are applied, and so on.
This leads to $\ell\cdot N_{\ref{algorithm:powerIterationRay}}$ steps of 
Algorithm~\ref{algorithm:powerIterationRay} in order to compute a starting tensor
for Algorithm~\ref{algorithm:powerIterationSquaring}, where $\ell$ is the number of 
loop iterations which is determined adaptively. 
\begin{algorithm}
\KwData{Tensor $\mathbf a$; number $N_{\ref{algorithm:powerIterationRay}}$ of
Algorithm~\ref{algorithm:powerIterationRay}-steps in each loop; upper bound $\varepsilon$ for the relative truncation error in Algorithm~\ref{algorithm:powerIterationSquaring}}
$\mathbf x := \mathbf a$\;
\SetKwRepeat{Do}{do}{while}
\Do {maximum number of iterations not yet reached}
{
$N_{\ref{algorithm:powerIterationRay}}$ steps of Algorithm~\ref{algorithm:powerIterationRay} (including truncations) on $\mathbf x$\;
Store the last iterate of Algorithm~\ref{algorithm:powerIterationRay} in $\mathbf x$\;
\label{line:storeLastIterate}
Apply Algorithm~\ref{algorithm:powerIterationSquaring} (including truncations) to $\mathbf x$\;
\If {truncation error was always smaller than $\varepsilon$}
{
Store the last iterate of Algorithm~\ref{algorithm:powerIterationSquaring} in $\mathbf x$\;\label{line:lastIterate}
\Return the resulting $\|\mathbf a\|_\infty$-approximation of Algorithm~\ref{algorithm:powerIterationSquaring}\;
}
}
\If {maximum number of iterations has been reached}
{
\Return the $\|\mathbf a\|_\infty$-approximation corresponding to the last iterate
$\mathbf x$ in line~\ref{line:storeLastIterate}\;
}
\caption{An adaptive combination of Algorithm~\ref{algorithm:powerIterationRay} and 
Algorithm~\ref{algorithm:powerIterationSquaring} which prevents too large truncation 
errors in Algorithm~\ref{algorithm:powerIterationSquaring}: first a starting tensor 
is computed by $\ell\cdot N_{\ref{algorithm:powerIterationRay}}$ steps of 
Algorithm~\ref{algorithm:powerIterationRay}, where $\ell$ is the number of loop 
iterations. Then Algorithm~\ref{algorithm:powerIterationSquaring} is applied to the
starting tensor until the iterates differ only by less than some prescribed relative 
accuracy.}
\label{algorithm:adaptive}
\end{algorithm}

Our experiments in Section~\ref{subsection:combination} show that this approach seems 
to work well for many tensors: in Figure~\ref{figure:histogramSQAD} we show the 
convergence rates of Algorithms~\ref{algorithm:powerIterationSquaring} and 
\ref{algorithm:adaptive} applied to 1000 random HT-tensors, for which 
Algorithm~\ref{algorithm:powerIterationSquaring} converges in $942/1000$ cases and
Algorithm~\ref{algorithm:adaptive} in $986/1000$ cases with a convergence rate far below
$0.25$.
Furthermore, we give an example for which 
Algorithm~\ref{algorithm:powerIterationSquaring} does not converge to a good 
approximation of $\|\mathbf a\|_\infty$, whereas 
Algorithm~\ref{algorithm:adaptive} performs well,
cf. Figure~\ref{figure:noConvergenceSquaring}.
For our experiments with Algorithm~\ref{algorithm:adaptive} we chose $N_3 = 10$.

\section{Finding extreme entries through a binary search}\label{section:finding}
Some applications may require not only an estimation of the maximum norm 
$\|\mathbf a\|_\infty$ but also an index where the corresponding extremum is attained in the tensor
$\mathbf a\in\mathbb{R}^I$, $I = \bigtimes_{\mu = 1}^d I_\mu$, which means that 
$\argmax_{\mathbf i\in I}|\mathbf a[\mathbf i]|$ is to be computed.

Using Algorithm~\ref{algorithm:adaptive}, this task can be accomplished by a 
binary search, cf. \cite{Knuth398}: in each step of the binary search the tensor $\mathbf a$ 
is divided into halves 
$\mathbf a_{1,\mu}$, $\mathbf a_{2,\mu}$ w.r.t. one of its directions $\mu = 1,\ldots ,d$ and the maximum
norms $\|\mathbf a_{1,\mu}\|_\infty$, $\|\mathbf a_{2,\mu}\|_\infty$ are approximated by 
Algorithm~\ref{algorithm:adaptive}. The binary search then continues on
\begin{equation}\label{equation:binarySearchRecursion}
\begin{array}{ccc}
\mathbf a_{1,\mu} & \qquad\mbox{if}\qquad & \|\mathbf a_{1,\mu}\|_\infty \geq \|\mathbf a_{2,\mu}\|_\infty,\\
\mathbf a_{2,\mu} & \qquad\mbox{if}\qquad & \|\mathbf a_{2,\mu}\|_\infty \geq \|\mathbf a_{1,\mu}\|_\infty.
\end{array}
\end{equation}
If no clear decision can be made, i.e. if $\|\mathbf a_{1,\mu}\|_\infty \approx \|\mathbf a_{2,\mu}\|_\infty$, 
it may
be advantageous to first continue with another tensor direction $\nu = 1,\ldots ,d$, $\nu\neq \mu$.

More precisely, we start with Algorithm~\ref{algorithm:adaptive}
applied on the entire tensor $\mathbf a$ and keep the last iterate $\mathbf a^{(\ell)}$
(i.e. $\mathbf x$ from line~\ref{line:lastIterate} in Algorithm~\ref{algorithm:adaptive}).
Then the
subsequent computations of $\|\mathbf a_{1,\mu}\|_\infty$ and $\|\mathbf a_{2,\mu}\|_\infty$
in the binary search can be carried out very fast by using the estimator
from line~\ref{line:alphaSquaring} of 
Algorithm~\ref{algorithm:powerIterationSquaring} on the corresponding
parts of $\mathbf a^{(\ell)}$.
The resulting procedure is shown in Algorithm~\ref{algorithm:binarySearch}.
\begin{algorithm}
\KwData{tensor $\mathbf a\in\mathbb{R}^I$, $I = \bigtimes_{\mu = 1}^d I_\mu$}
\KwResult{index $\mathbf j = (j_1,\ldots ,j_d)\in I$ such that $\|\mathbf a\|_\infty = |\mathbf a[\mathbf j]|$}
\eIf {$\mathbf a$ is an elementary tensor}
{
    Determine the $\argmax$-index $\mathbf j$ as described in Section~\ref{section:rank1}\;
}
{
Apply Algorithm~\ref{algorithm:adaptive} to $\mathbf a$ and keep the last iterate
$\mathbf a^{(\ell)}$\;\label{line:precomputation}
\eIf {$\mathbf a^{(\ell)}$ is an elementary tensor}
{
Determine the $\argmax$-index $\mathbf j$ as described in Section~\ref{section:rank1}\;
}
{
\While{$\# I > 1$ {\normalfont\bf for} $\mu = 1,\ldots ,d$ }
{
\eIf{ $\#I_\mu > 1$}
{
Cut $\mathbf a$ into two slices of (nearly) same size
\begin{equation*}
\mathbf a_{1,\mu} := \mathbf a\mid_{I_1\times\cdots\times I^{(1)}_\mu\times\cdots\times I_d},\qquad
\mathbf a_{2,\mu} := \mathbf a\mid_{I_1\times\cdots\times I^{(2)}_\mu\times\cdots\times I_d}
\end{equation*}\\
Approximate $\|\mathbf a_{1,\mu}\|_\infty$, $\|\mathbf a_{2,\mu}\|_\infty$ by 
Algorithm~\ref{algorithm:powerIterationSquaring} 
(on $\mathbf a^{(\ell)}$)\;
Continue the binary search according to \eqref{equation:binarySearchRecursion}\;
}
{
$j_\mu := i_\mu$, where $I_\mu = \{i_\mu\}$\;
}
}
}
}
\caption{Binary search for a tensor index $\mathbf j = \argmax_{\mathbf i\in I}|\mathbf a[\mathbf i]|$ which corresponds to the maximum norm $\|\mathbf a\|_\infty$.}
\label{algorithm:binarySearch}
\end{algorithm}

\paragraph*{Complexity}
Algorithm~\ref{algorithm:binarySearch} first determines an eigenvector approximation
$a^{(\ell)}$, cf. line~\ref{line:precomputation}, the computational effort of which can 
be bounded by
$\mathcal{O}(d(nr^2 + r^6))$ flops, cf. 
Sections~\ref{subsection:htFormat} and \ref{section:power},
when all tensor directions are assumed to be of the same size $n = \#I_\mu$, $\mu = 1,\ldots ,d$,
and the components of the HT-ranks are bounded by $r\in\mathbb{N}$.
In each step of the binary search the tensor is divided into halves, only one of which is retained. 
Since the tensor has $n^d$ entries, this yields a number of
\begin{equation}\label{equation:numberOfIterationsSearch}
\log_2(n^d) = d\log_2(n)\qquad\mbox{iterations}
\end{equation}
for Algorithm~\ref{algorithm:binarySearch} (total number of iterations for the 
{\bf while}/{\bf for}-loop.
By computing the eigenvector approximation $\mathbf a^{(\ell)}$ in advance, we only need 
one further Hadamard product and one dot product for the computation of 
$\|\mathbf a_{1,\mu}\|_\infty$ and $\|\mathbf a_{2,\mu}\|_\infty$ in each iteration.
The overall complexity of Algorithm~\ref{algorithm:binarySearch} can be estimated by
\begin{equation}\label{equation:complexitySearchSum}
\underbrace{\mathcal{O}(d(nr^2 + r^6))}_{\mbox{\scriptsize Computation of $a^{(\ell)}$}}
+ \underbrace{d\log_2(n)}_{\mbox{\scriptsize number of iterations}} \cdot\;
\underbrace{\mathcal{O}(d(nr^2 + r^6))}_{\mbox{\scriptsize one Hadamard product + one dot product}}
\end{equation}
flops, i.e. a total complexity of
\begin{equation}\label{equation:complexitySearch}
\mathcal{O}(d^2\log_2(n)(nr^2 + r^6))\qquad\mbox{flops}.
\end{equation}
Our experiments in Section~\ref{subsection:experimentsArgmax} confirm the complexity
estimation \eqref{equation:complexitySearch} for Algorithm~\ref{algorithm:binarySearch}:
Figure~\ref{figure:argmaxBinarySearch} shows an $\mathcal{O}(d^2)$-dependence for the
runtime of Algorithm~\ref{algorithm:binarySearch}, while 
Algorithm~\ref{algorithm:adaptive} has linear complexity in $d$.

However, the number \eqref{equation:numberOfIterationsSearch} of iterations of 
Algorithm~\ref{algorithm:binarySearch} can often be reduced for two reasons:
\begin{enumerate}[(a)]
\item Since the iterates $\mathbf a^{(j+1)}\in\mathbb{R}^I$ of\label{enumerate:eigenvecLowRank1}
Algorithms \ref{algorithm:powerIterationRay}
and \ref{algorithm:powerIterationSquaring} correspond to approximations of eigenvectors of the diagonal matrix
$D\in\mathbb{R}^{I\times I}$, cf. \eqref{equation:tensorAsDiagMatrix},
\eqref{equation:maximumNormAsEigenvalue}, they are typically of very low rank in the end.
As soon as the iterates are rank-$1$ tensors, the $\argmax$ can directly be computed without
further iterations, cf. Section~\ref{section:rank1}.
\item For the same reason as in \eqref{enumerate:eigenvecLowRank1} the iterates
$\mathbf a^{(j+1)}$ typically contain many zero entries which is reflected in zero rows at the 
leaf vertices of the HT-format: any zero row in $U^{(\{\mu\})}$, $\mu = 1,\ldots ,d$, cf.
Definition~\ref{definition:htFormat}, 
can be removed, which reduces the size of the corresponding tensor direction and with that the
number of further iterations.
\end{enumerate}

\paragraph*{Parallelization}
In \cite{GrLoe18} we present parallel algorithms for tensor operations which lead to 
runtimes of 
$\mathcal{O}(\log_2(d)(nr^2 + r^4))$ and $\mathcal{O}(\log_2(d)(nr^2 + r^6))$
instead of $\mathcal{O}(d(nr^2 + r^4))$ and $\mathcal{O}(d(nr^2 + r^6))$ and
thus to a runtime of
\begin{equation*}
\mathcal{O}(d\log_2(n)\log_2(d)(nr^2 + r^6))
\end{equation*}
instead of \eqref{equation:complexitySearch} for the corresponding parallel version of 
Algorithm~\ref{algorithm:binarySearch}, i.e. a nearly linear dependence on $d$ when using
$d$ processes, cf. \cite{Loe19}.

One could go even further and parallelize the {\bf for}-loop in 
Algorithm~\ref{algorithm:binarySearch}: for each direction $\mu = 1,\ldots ,d$ the tensor is
split into $\mathbf a_{1,\mu}$ and $\mathbf a_{2,\mu}$ for which 
$\|\mathbf a_{1,\mu}\|_\infty$ and $\|\mathbf a_{2,\mu}\|_\infty$ are
estimated by 
line~\ref{line:alphaSquaring} of Algorithm~\ref{algorithm:powerIterationSquaring},
which needs one Hadamard product and one dot product.
The Hadamard products for the splittings w.r.t. the 
directions $\mu = 1,\ldots ,d$ can be obtained by splitting the respective Hadamard product
of the whole tensor; the dot products can simultaneously be computed for all splittings occurring
during the {\bf for}-loop.
This approach would lead to a runtime bounded by
\begin{equation*}
\mathcal{O}(\log_2(n)\log_2(d)(nr^2 + r^6)),
\end{equation*}
i.e. a sub-linear dependence on $d$ and $n$ when using $d$ processes. 

\section{Numerical experiments}\label{section:experiments}
We present several numerical experiments for the algorithms discussed in this article.
The experiments indicate that the methods which we propose for the maximum norm estimation (Algorithm~\ref{algorithm:adaptive}) and for the search of the maximizing index
(Algorithm~\ref{algorithm:binarySearch})
seem to work for many tensors which occur in
practical applications. However, in Section~\ref{subsection:noConvergence} we also 
give an example which shows the limit of our method.

Most of our experiments are based on the following two classes of tensors of entirely different
structure, for which we know the exact value of $\|\mathbf a\|_\infty$ as well as an index
$\mathbf i$ with $\|\mathbf a\|_\infty = |\mathbf a[\mathbf i]|$, by which we can  
evaluate the output of our algorithms:
\begin{enumerate}
\item By $\rand(d,n,r)$ we refer to tensors of order $d\in\mathbb{N}$ in HT-representation
with each tensor direction of size $n\in\mathbb{N}$ and all components
of the HT-rank equal to $r$ 
(except for $r_D = 1$, cf. Definition~\ref{definition:htFormat}). 
The HT-data is generated as uniform random samples from $[-1.5,1.5]$, 
where the leaf frames
$U^{(\{\mu\})}$, $\mu = 1,\ldots ,d$, are generated by randomly repeating two random
row vectors of size $r$ such that $\argmax_{\mathbf i\in I}|\mathbf a[\mathbf i]|$ 
can easily be computed exactly for $d=16$, since only
$2^{16}$ different entries occur.
\item By $\cheb(d,n)$ we denote the tensor which results from the discretization of the 
Chebyshev polynomial $T_4$ of degree $4$ on the interval $[-1,1]$ with
$N:=n^d$ equidistant grid points
\begin{equation*}
x_i = -1 + \frac{2(i-1)}{N-1},\qquad i = 1,\ldots ,N.
\end{equation*}
This tensor can easily be assembled in the HT-format
with HT-rank bounded by $4+1=5$, cf. \cite{Gr10tensorization}.
With $\cheb(d,n)$ we thus have at hand example tensors $\mathbf a\in\mathbb{R}^I$ of high tensor 
order $d$ and possibly large mode sizes $n$ with
\begin{equation*}
\|\mathbf a\|_\infty = 1 = |\mathbf a[ 1,\ldots ,1]|.
\end{equation*}
\end{enumerate}
In our experiments we truncate back to the original representation rank of the tensor $\mathbf a$
after each computation of Hadamard products and after each HT-QR decomposition.
However, higher accuracies of the maximum norm estimation can be obtained by truncating 
with higher accuracy, which 
increases the runtime of the algorithms.
Note that for any positive truncation error one can construct counterexamples 
which need higher accuracy
in order to reach a good approximation of $\|\mathbf a\|_\infty$, cf. Section~\ref{subsection:noConvergence}.

\subsection{Power iteration (Algorithm~\ref{algorithm:powerIterationImproved})}
Figure~\ref{figure:powerIterationSimple} shows the relative errors of 
Algorithm~\ref{algorithm:powerIterationImproved} applied to two instances of the class
$\rand(16,100,5)$ and to the tensor $\cheb(16,100)$. 
The convergence rates for $\cheb(16,100)$ as well as for one of the 
$\rand(16,100,5)$-tensors tend to one as indicated by \eqref{equation:badConvergence}.
The other instance of $\rand(16,100,5)$ shows a linear convergence with a convergence 
rate of $\approx 0.52$.

\begin{figure}
\begin{subfigure}[b]{0.48\textwidth}
\input{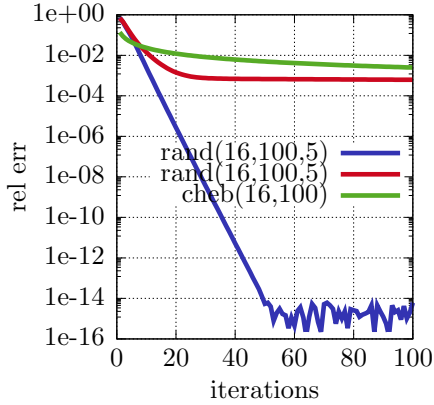}
\caption{Relative errors of Algorithm~\ref{algorithm:powerIterationImproved}
applied to two tensors of type $\rand(16,100,5)$ and to the tensor $\cheb(16,100)$. 
The convergence can be very slow with a convergence rate tending to one,
cf. \eqref{equation:badConvergence}.}
\label{figure:powerIterationSimple}
\end{subfigure}\quad
\begin{subfigure}[b]{0.48\textwidth}
\input{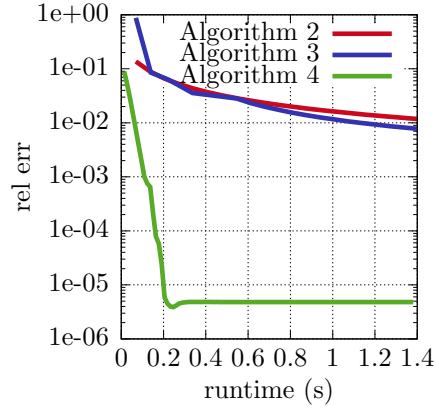}
\caption{Algorithm~\ref{algorithm:powerIterationSquaring} applied to the tensor 
$\cheb(16,100)$ together with the corresponding results from Figure~\ref{subfigure:powerIterationProjectionCheb} (Algorithms~\ref{algorithm:powerIterationImproved} and \ref{algorithm:powerIterationRay}).\\}
\label{figure:powerIterationSquaring}
\end{subfigure}
\caption{Algorithms~\ref{algorithm:powerIterationImproved} and \ref{algorithm:powerIterationSquaring} applied to example tensors.}
\end{figure}

In order to get an impression of how likely the slow convergence behavior
\eqref{equation:badConvergence} occurs for 
Algorithm~\ref{algorithm:powerIterationImproved}, we measured the convergence rates
of Algorithm~\ref{algorithm:powerIterationImproved} for 1000 random tensors of 
$\rand(16,1000,5)$. 
The resulting histogram is depicted in Figure~\ref{subfigure:histogramPI}:
in $24.6\%$ of the cases a convergence rate larger than $0.95$ was measured; 
a convergence rate of $0.75$ or better was observed in $19.2\%$ of the cases.
These results indicate that the sole power iteration in general converges
quite slowly.
\begin{figure}
\begin{subfigure}[b]{0.48\textwidth}
\input{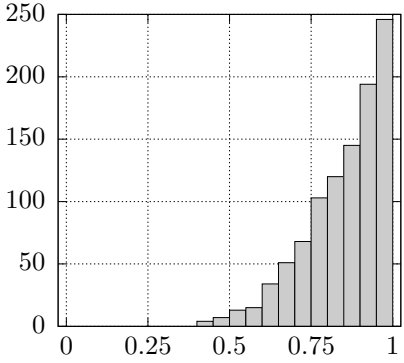}
\caption{Algorithm~\ref{algorithm:powerIterationImproved}}
\label{subfigure:histogramPI}
\end{subfigure}
\begin{subfigure}[b]{0.48\textwidth}
\input{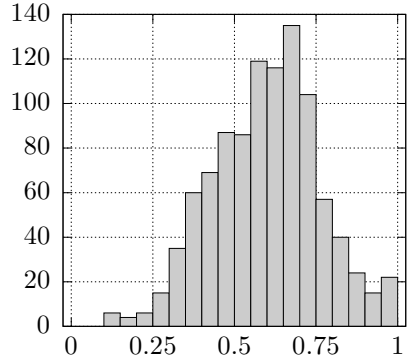}
\caption{Algorithm~\ref{algorithm:powerIterationRay} ($k=5$)}
\label{subfigure:histogramPR}
\end{subfigure}
\caption{Convergence rates of Algorithm~\ref{algorithm:powerIterationImproved} and
Algorithm~\ref{algorithm:powerIterationRay} (with $k=5$) for 1000 random tensors of the class
$\rand(16,1000,5)$. We used the same 1000 tensors for both algorithms.}
\label{figure:histogramPIPR}
\end{figure}

\subsubsection{Measurement of the convergence rate}\label{subsection:measurementRate}
The convergence rates in Figures~\ref{figure:histogramPIPR} and \ref{figure:histogramSQAD}
have been taken as the quotient $\mathrm{error}_{j}/\mathrm{error}_{j-1}$, where $j$
is the last step with a relative error larger than $10^{-12}$.
For Algorithms~\ref{algorithm:powerIterationImproved}, \ref{algorithm:powerIterationRay}, \ref{algorithm:powerIterationSquaring} we performed 40 steps in total, whereas
the number of steps in Algorithm~\ref{algorithm:adaptive} is chosen adaptively, cf.
Section~\ref{section:squaring}.

\subsection{Power iteration with Rayleigh-Ritz method (Algorithm~\ref{algorithm:powerIterationRay})}
In Figure~\ref{figure:powerIterationProjection} we apply 
Algorithm~\ref{algorithm:powerIterationRay} to the slowly converging tensors from
Figure~\ref{figure:powerIterationSimple}. It turns out that in one case
(the slowly converging instance of $\rand(16,100,5)$ from 
Figure~\ref{figure:powerIterationSimple}) we get linear convergence of 
Algorithm~\ref{algorithm:powerIterationRay}, whereas the convergence rate for
$\cheb(16,100)$ still tends to one, i.e. the convergence behavior does not improve in 
this case.
\begin{figure}
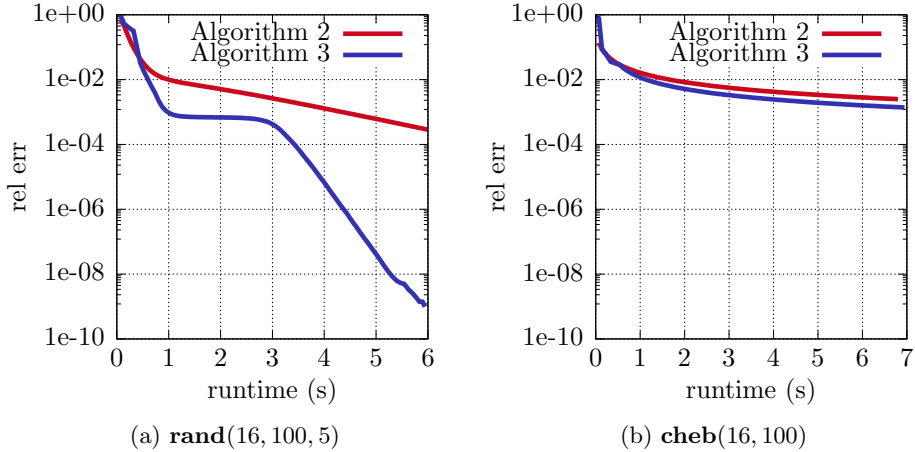

\begin{subfigure}[b]{0.48\textwidth}
\input{plot-rand-16-100-5-PR.tex}
\caption{$\rand(16,100,5)$}
\label{subfigure:powerIterationProjectionRand}
\end{subfigure}
\begin{subfigure}[b]{0.48\textwidth}
\input{plot-cheb-16-100-5-PR.tex}
\caption{$\cheb(16,100)$}
\label{subfigure:powerIterationProjectionCheb}
\end{subfigure}
\caption{Improvement achieved by Algorithm~\ref{algorithm:powerIterationRay} (with $k=5$)
applied to 
the slowly converging tensors from Figure~\ref{figure:powerIterationSimple}: for the
$\rand(16,100,5)$-tensor Algorithm~\ref{algorithm:powerIterationRay} now shows linear
convergence with rate $\approx 0.71$, whereas the convergence behavior for 
$\cheb(16,100)$ is still bad (convergence rate tending to one).}
\label{figure:powerIterationProjection}
\end{figure}

In order to get an idea of how much Algorithm~\ref{algorithm:powerIterationRay} improves
Algorithm~\ref{algorithm:powerIterationImproved}, we applied 
Algorithm~\ref{algorithm:powerIterationRay} to the same 1000 tensors which have been
used for the histogram in Figure~\ref{subfigure:histogramPI}, cf. 
Figure~\ref{subfigure:histogramPR}.
Based on Figure~\ref{figure:histogramPIPR}, 
Algorithm~\ref{algorithm:powerIterationRay} seems to yield a notable acceleration of
the convergence speed. Notice, however, that Figure~\ref{figure:histogramPIPR} only
reflects the the class $\rand(16,1000,5)$ of random tensors.

For our experiments we used a Rayleigh-Ritz method w.r.t. the latest $k=5$ iterates in 
each step, i.e. the subspace \eqref{equation:rayritzSubspace} has dimension $5$. 
In order to compare Algorithms~\ref{algorithm:powerIteration} and 
\ref{algorithm:powerIterationRay}, we applied the Rayleigh-Ritz method in each step 
of Algorithm~\ref{algorithm:powerIterationRay}.
In all of our experiments
$k=5$ has proven to be a suitable choice:
the average relative error for 100 instances of 
$\rand(d,n,4)$ after 100 steps of power iteration and subsequent Rayleigh-Ritz method 
with subspace dimension
$k$ was minimized for $k=5$ or $k=6$, regardless of the tensor size.
If $k$ is chosen too small, the resulting Rayleigh-Ritz method has only
little effect. If, on the other hand, $k$ is chosen too large, the 
orthogonalization (HT-QR followed by a truncation, cf. Section~\ref{subsection:HTQR}) 
might not be exact enough which seems to
lead to even higher errors than without the Rayleigh-Ritz method.

Compared to the overall computational work of the power iteration, the effort for the subsequent 
Rayleigh-Ritz method is small.

\subsection{Power iteration with squaring (Algorithm~\ref{algorithm:powerIterationSquaring})}\label{subsection:powerIterationSquaring}
A significant acceleration of the convergence can be achieved by using 
Algorithm~\ref{algorithm:powerIterationSquaring} described in Section~\ref{section:squaring}.
The resulting relative errors for the tensor $\cheb(16,100)$ are depicted in 
Figure~\ref{figure:powerIterationSquaring} in comparison with the errors from 
Figure~\ref{subfigure:powerIterationProjectionCheb}:
Algorithm~\ref{algorithm:powerIterationSquaring} converges almost
instantaneously to a relative error of about $5\cdot 10^{-6}$, where it stagnates due to
truncations of the iterates down to lower HT-representation ranks.

Figure~\ref{subfigure:histogramSQ} shows the corresponding histogram for 
Algorithm~\ref{algorithm:powerIterationSquaring} applied to the same 1000 tensors used
for Figure~\ref{figure:histogramPIPR}: as predicted in Section~\ref{section:squaring},
cf. \eqref{equation:linearConvergenceSquaring}, in most cases we observe a convergence
rate of less than $0.5$. However, in $5.8\%$ of the cases the convergence rate is close
to one, which reflects those cases where Algorithm~\ref{algorithm:powerIterationSquaring}
does not converge (i.e. stagnates early\footnote{A stagnation of the algorithm at some relative error above $10^{-12}$ is counted as convergence rate $\approx 1.0$ in our experiments, cf. Section~\ref{subsection:measurementRate}.}) because of the truncation errors becoming to large, cf.
\eqref{equation:hadamardProductStable}, \eqref{equation:hadamardProductLessStable}.
This stability issue of Algorithm~\ref{algorithm:powerIterationSquaring} can be 
improved by using Algorithm~\ref{algorithm:adaptive} which is an adaptive combination
of Algorithms~\ref{algorithm:powerIterationRay} 
and \ref{algorithm:powerIterationSquaring}:
first a good enough starting vector is computed by the more stable 
Algorithm~\ref{algorithm:powerIterationRay} and then the fast convergence of
Algorithm~\ref{algorithm:powerIterationSquaring} is exploited using this starting vector,
cf. Section~\ref{section:squaring}.
\begin{figure}
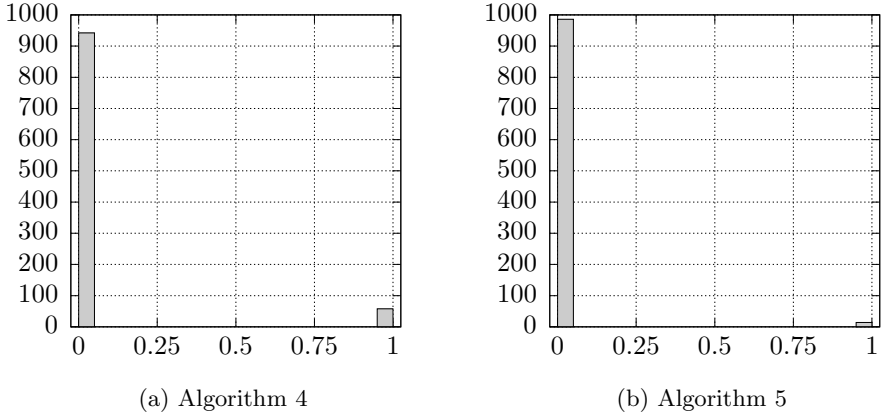

\begin{subfigure}[b]{0.48\textwidth}
\input{plot-hist-SQ.tex}
\caption{Algorithm~\ref{algorithm:powerIterationSquaring}}
\label{subfigure:histogramSQ}
\end{subfigure}
\begin{subfigure}[b]{0.48\textwidth}
\input{plot-hist-AD.tex}
\caption{Algorithm~\ref{algorithm:adaptive}}
\label{subfigure:histogramAD}
\end{subfigure}
\caption{Convergence rates of Algorithms~\ref{algorithm:powerIterationSquaring} and
\ref{algorithm:adaptive} for the 
same 1000 tensors of $\rand(16,1000,5)$ as in Figure~\ref{figure:histogramPIPR}:
we measured very small convergence rates far below $0.25$. In $58/1000$ cases 
Algorithm~\ref{algorithm:powerIterationSquaring} does not converge (i.e. stagnates early) because of too large
truncation errors, cf. \eqref{equation:hadamardProductLessStable}. For $44$ 
(i.e. $75\%$) of these tensors, however,
convergence can be achieved by using Algorithm~\ref{algorithm:adaptive} instead.
In Algorithm~\ref{algorithm:adaptive} we applied $N_3 = 10$ more steps of 
Algorithm~\ref{algorithm:powerIterationRay} in each cycle.}
\label{figure:histogramSQAD}
\end{figure}

\subsection{Power iteration with squaring and stabilization (Algorithm~\ref{algorithm:adaptive})}
\label{subsection:combination}
In Figure~\ref{figure:noConvergenceSquaring} we give an example of a tensor for which 
Algorithm~\ref{algorithm:powerIterationSquaring} does not converge,
whereas Algorithm~\ref{algorithm:adaptive} (including 2 steps of 
Algorithm~\ref{algorithm:powerIterationRay} for the computation of the starting 
tensor) converges.
The example tensor of Figure~\ref{figure:noConvergenceSquaring} was found as random
tensor\footnote{The random tensor of Figure~\ref{figure:noConvergenceSquaring} is of size
$2\times 2\times 4\times 4\times 2\times 2\times 6\times 8$ and was generated by the 
{\it htucker} MATLAB toolbox which uses the MATLAB function {\tt randn} to generate
normally perturbed random data for an HT-representation.}
generated by the MATLAB toolbox {\it htucker} \cite{Kressner14}.

\begin{figure}
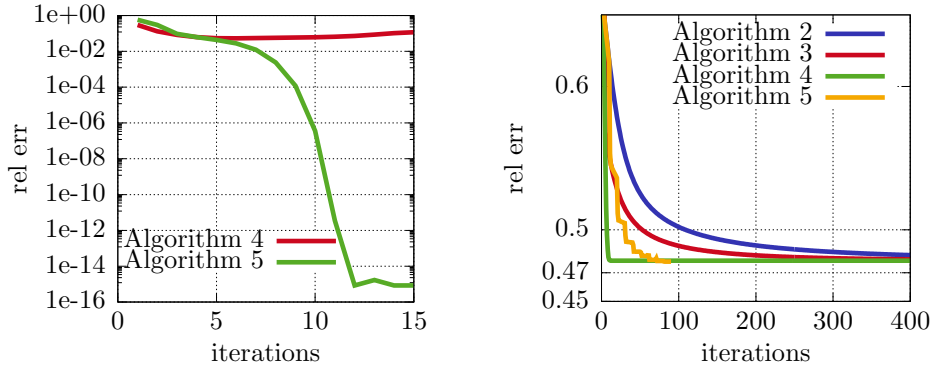

\begin{subfigure}[b]{0.48\textwidth}
\input{plot-counter.tex}
\caption{An example tensor (order $d=8$ with $12\,288$ entries) for which Algorithm~\ref{algorithm:powerIterationSquaring} does not
converge whereas Algorithm~\ref{algorithm:adaptive} (including only two steps of 
Algorithm~\ref{algorithm:powerIterationRay}) does.\\} 
\label{figure:noConvergenceSquaring}
\end{subfigure}\quad
\begin{subfigure}[b]{0.48\textwidth}
\input{plot-no-convergence.tex}
\caption{A counterexample $\mathbf a\in\mathbb R^I$ defined by \eqref{equation:counterexample3a}, \eqref{equation:counterexample3b}, for which none of our algorithms approximates well the maximum norm $\|\mathbf a\|_\infty$: the relative error of all methods stays above $0.47$.}
\label{figure:noConvergence}
\end{subfigure}
\caption{Counterexamples}
\end{figure}

In Figure~\ref{subfigure:histogramAD} we apply Algorithm~\ref{algorithm:adaptive} to
the same 1000 tensors as in Figure~\ref{subfigure:histogramSQ} 
(and in Figure~\ref{figure:histogramPIPR}): we now get convergence for $75\%$ 
of the cases where Algorithm~\ref{algorithm:powerIterationSquaring} did not converge.

\subsection{An example where the algorithms fail}\label{subsection:noConvergence}
Based on \eqref{equation:counterexample} we can construct tensors $\mathbf a\in\mathbb{R}^I$,
$I = \bigtimes_{\mu = 1}^d I_\mu$, for which our algorithms fail: the tensor entry
$\mathbf a[\mathbf i]$ which maximizes $|\mathbf a[\mathbf i]|$ is deleted by truncations which
results in the algorithms no longer converging to $\|\mathbf a\|_\infty$.
As an example take a tensor $\mathbf a\in\mathbb{R}^I$, $I = \bigtimes_{\mu = 1}^d I_\mu$, with
$d = 10$ and mode sizes $\# I_\mu = n$, $\mu = 1,\ldots ,d$, defined by
\begin{equation}\label{equation:counterexample3a}
\mathbf a := 
\left(
\begin{array}{r}
\mbox{\tt rand}[0.91,1.0]\\
\vdots\\
\mbox{\tt rand}[0.91,1.0]
\end{array}
\right)\otimes\cdots\otimes
\left(
\begin{array}{r}
\mbox{\tt rand}[0.91,1.0]\\
\vdots\\
\mbox{\tt rand}[0.91,1.0]
\end{array}
\right)
\end{equation}
at all indices $\mathbf i\neq (1,\ldots ,1)$ and 
\begin{equation}\label{equation:counterexample3b}
\mathbf a[1,\ldots ,1] := 1.9.
\end{equation}
The elementary tensor \eqref{equation:counterexample3a} has HT-rank $(1,\ldots ,1)$ and the 
extra definition \eqref{equation:counterexample3b} increases the components of the HT-rank to $2$.
The resulting tensor $\mathbf a$ contains entries in $[0.1,1.0]$ at all indices 
$\mathbf i\neq (1,\ldots ,1)$ and attains its maximum $1.9$ at $(1,\ldots ,1)$, i.e.
$\|\mathbf a\|_\infty = 1.9$.
As shown in Figure~\ref{figure:noConvergence}, none of the methods presented in this article 
approximates well the maximum norm $\|\mathbf a\|_\infty$: the relative errors stay above $0.47$
which corresponds to a computed maximum norm of less than $1.007$ instead of $1.9$.

\subsection{Finding the maximizing index by a binary search (Algorithm~\ref{algorithm:binarySearch})}\label{subsection:experimentsArgmax}
In Figure~\ref{figure:argmaxBinarySearch} the runtimes of 
Algorithms~\ref{algorithm:adaptive} and \ref{algorithm:binarySearch}
are displayed, once for varying tensor order $d$ and once for varying mode sizes $n$.
As expected, we find the runtime of Algorithm~\ref{algorithm:adaptive} to depend 
linearly on the tensor order $d$, whereas the runtime of 
Algorithm~\ref{algorithm:binarySearch} in the end grows quadratically with $d$, 
cf. \eqref{equation:complexitySearch}.
Also the dependence of the runtime on $n$ behaves as expected: finally both
algorithms seem to grow linearly with $n$.
\begin{figure}
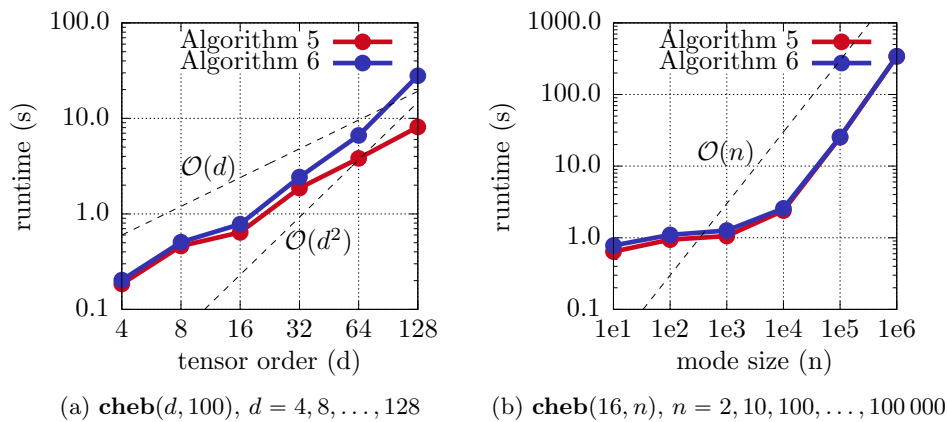

\begin{subfigure}[b]{0.48\textwidth}
\input{plot-argmax-d.tex}
\caption{$\cheb(d,100)$, $d = 4,8,\ldots ,128$}
\end{subfigure}
\begin{subfigure}[b]{0.48\textwidth}
\input{plot-argmax-n.tex}
\caption{$\cheb(16,n)$, $n = 2,10,100,\ldots ,100\,000$}
\end{subfigure}
\caption{Runtimes of Algorithms~\ref{algorithm:adaptive} and 
\ref{algorithm:binarySearch} for the tensor $\mathbf a = \cheb(d,n)$ with tensor 
orders $d = 4,8,\ldots,128$ and mode sizes $n = 10,100,\ldots ,1\,000\,000$. 
For all values of $d$ and $n$ the corresponding relative approximation error of 
$\|\mathbf a\|_\infty$ lies between $2\cdot 10^{-10}$ and $5\cdot 10^{-4}$.}
\label{figure:argmaxBinarySearch}
\end{figure}

\section{Conclusions}
In this article we have shown how extreme entries in low-rank tensors
can be approximated using a power iteration. The resulting algorithms can
be applied for any low-rank tensor format which meets the requirements
stated at the end of Section~\ref{section:tensorformats}.
We presented two approaches how to accelerate the generally slowly
converging power iteration. The algorithm which we suggest is an adaptive
combination of these accelerated methods.
A binary search makes it possible to also find an according extreme
entry of the tensor, i.e. find a tensor index which maximizes the
absolute value of tensor entries.
Our experiments in 
Section~\ref{section:experiments} indicate that our algorithms work
well for many tensors, cf. Figures~\ref{figure:histogramSQAD} and
\ref{figure:argmaxBinarySearch}.
Even for those tensors of Figure~\ref{figure:histogramSQAD} for which
the convergence of
Algorithms~\ref{algorithm:powerIterationSquaring} or 
\ref{algorithm:adaptive} stagnates early, the maximum relative error
is $4\cdot 10^{-3}$, which would still be of good use in many
applications.

The maximum norm estimator of our algorithm approximates the maximum
norm from below (i.e. yields a lower bound), cf.
Lemma~\ref{lemma:improvedEstimator}.
In order to use a \emph{branch and bound} algorithm it would be 
interesting to have useful upper bounds for the
maximum norm.

\bibliographystyle{siam}
\bibliography{refs}

\end{document}